\documentclass[10pt]{article}

\usepackage{etex}
\usepackage{subcaption}
\usepackage[utf8]{inputenc}
\usepackage{a4wide}
\usepackage{graphicx}
\usepackage{wrapfig}
\usepackage[hmarginratio={1:1},     
vmarginratio={1:1},     
textwidth=16cm,        
textheight=22cm,
heightrounded,]{geometry}
\usepackage{amsmath,accents}
\usepackage{amsthm}
\usepackage{amssymb}
\usepackage{mathrsfs, mathtools}
\usepackage{overpic}
\usepackage{tocloft}
\usepackage[all,cmtip]{xy}

\let\emph\undefined
\newcommand{\emph}[1]{\textsl{#1}}

\usepackage{hyperref}

\numberwithin{equation}{section}
\usepackage{mathtools}
\mathtoolsset{showonlyrefs}
\numberwithin{equation}{section}

\newtheoremstyle{style1}
{13pt}
{13pt}
{}
{}
{\normalfont\bfseries}
{.}
{.5em}
{}

\theoremstyle{style1}

\newtheorem{definition}[equation]{Definition}

\newtheorem{remark}[equation]{Remark}

\newtheoremstyle{style2}
{13pt}
{13pt}
{\slshape}
{}
{\normalfont\bfseries}
{.}
{.5em}
{}

\theoremstyle{style2}

\newtheorem{lemma}[equation]{Lemma}
\newtheorem{theorem}[equation]{Theorem}
\newtheorem{proposition}[equation]{Proposition}

\def\hocolim{\mathrm{hocolim}}

\usepackage{tikz}
\usetikzlibrary{matrix,arrows,decorations.pathmorphing,shapes.geometric}
\usepackage{tikz-cd}
\usepackage{enumitem}

\usepackage{needspace}

\newcommand{\Grpd}{\catf{Grpd}}
\newcommand{\Map}{\operatorname{Map}}

\newcommand{\PBun}{\catf{PBun}}
\newcommand{\cat}[1]{\mathcal{#1}}
\newcommand{\ball}{\mathbb{B}}
\newcommand{\catf}[1]{{\mathsf{#1}}}
\let\P\undefined
\newcommand{\P}{\mathsf{P}}
\newcommand{\Aut}{\operatorname{Aut}}

\newcommand{\Hom}{\operatorname{Hom}}
\newcommand{\im}{\operatorname{im}}
\newcommand{\ev}{\operatorname{ev}}
\newcommand{\id}{\operatorname{id}}

\newcommand{\TwoVect}{{2\mathsf{Vect}_\mathbb{C}}}
\newcommand{\Op}{{\catf{Op}_\mathfrak{C}}}
\newcommand{\SymS}{{\catf{Sym}_\mathfrak{C}}}
\newcommand{\PBr}{\mathsf{PBr}}
\newcommand{\Cob}{{\catf{Cob}}}
\newcommand{\Cat}{\catf{Cat}}

\newcommand{\U}{\operatorname{U}}

\newcommand{\Set}{\catf{Set}}

\let\to\undefined
\newcommand{\to}{\longrightarrow}
\let\mapsto\undefined
\newcommand{\mapsto}{\longmapsto}

\newcommand{\Alg}{\catf{Alg}}
\newcommand{\opp}{\text{opp}}

\newcommand{\disk}{\mathbb{D}}
\newcommand{\comp}{\mathsf{C}}

\DeclareMathSymbol{\Phiit}{\mathalpha}{letters}{"08} 
\DeclareMathSymbol{\Psiit}{\mathalpha}{letters}{"09}
\DeclareMathSymbol{\Sigmait}{\mathalpha}{letters}{"06}
\DeclareMathSymbol{\Xiit}{\mathalpha}{letters}{"04}
\DeclareMathSymbol{\Piit}{\mathalpha}{letters}{"05}\let\Pi\undefined\newcommand{\Pi}{\Piit}
\DeclareMathSymbol{\Gammait}{\mathalpha}{letters}{"00}
\DeclareMathSymbol{\Omegait}{\mathalpha}{letters}{"0A}
\DeclareMathSymbol{\Upsilonit}{\mathalpha}{letters}{"07}
\DeclareMathSymbol{\Thetait}{\mathalpha}{letters}{"02}

\let\Phi\undefined\newcommand{\Phi}{\Phiit}
\let\Sigma\undefined\newcommand{\Sigma}{\Sigmait}
\let\Psi\undefined\newcommand{\Psi}{\Psiit}
\let\Gamma\undefined\newcommand{\Gamma}{\Gammait}

\usepackage{booktabs}
\usepackage{tabularx}

\begin{document}

	\vspace*{-1.5cm}
	\begin{flushright}
		\small
		{\sf EMPG--19--02} \\
		{\sf [ZMP-HH/19-1]} \\
		\textsf{Hamburger Beiträge zur Mathematik Nr.~771}\\
		\textsf{January 2019}
	\end{flushright}
	
	\vspace{5mm}
	
	\begin{center}
		\textbf{\LARGE{The Little Bundles Operad}}\\
		\vspace{0.5cm}
		{\large Lukas Müller $^{a}$} \ \ and \ \ {\large Lukas Woike $^{b}$}
		
		\vspace{5mm}
		
		{\em $^a$ Department of Mathematics\\
			Heriot-Watt University\\
			Colin Maclaurin Building, Riccarton, Edinburgh EH14 4AS, U.K.}\\
		and {\em Maxwell Institute for Mathematical Sciences, Edinburgh, U.K.}\\
		Email: \ {\tt lm78@hw.ac.uk \ }
		\\[7pt]
		{\em $^b$ Fachbereich Mathematik\\ Universit\"at Hamburg\\
			Bereich Algebra und Zahlentheorie\\
			Bundesstra\ss e 55, \ D\,--\,20\,146\, Hamburg }\\
		Email: \ {\tt  lukas.jannik.woike@uni-hamburg.de\ }
	\end{center}
	\vspace{0.3cm}
	\begin{abstract}\noindent 
		Hurwitz spaces are homotopy quotients of the braid group action on the moduli space of principal bundles over a punctured plane. 
	By considering a certain model for this homotopy quotient we build an aspherical topological operad that we call the \emph{little bundles operad}. 
	As our main result, we describe this operad as a groupoid-valued operad in terms of generators and relations and prove that the categorical little bundles algebras are precisely 
braided $G$-crossed categories in the sense of Turaev. Moreover, we prove that the evaluation on the circle of a homotopical two-dimensional equivariant topological field theory yields a little bundles algebra up to coherent homotopy. \\[2ex] \textbf{Keywords:} operad, topological field theory, braid group, monoidal category, braided monoidal category
	\end{abstract}

	\tableofcontents

\section{Introduction}
Consider for $r\ge 0$ an $r$-ary operation $f \in E_2(r)$ of the little disks operad $E_2$ \cite{bv68,bv73}, i.e.\ an affine embedding of $r$ disks into another disk, and the groupoid $\PBun_G(\comp(f))$ of $G$-bundles over the closed complement $\comp(f)$ of the image of the embedding $f$. 
Then the (pure) braid group on $r$ strands acts on the space $\PBun_G(\comp(f))$. 
The homotopy quotient is known as a \emph{Hurwitz space}.
We consider a model
$W_2(r)$ for this homotopy quotient which, by restriction to the boundary circles, comes with a Serre fibration $W_2(r) \to \Map(\mathbb{S}^1,BG)^{r+1}$
to the $r+1$-fold product of the free loop space of the classifying space of $G$. This allows us to prove that the fibers of this Serre fibration, considered for varying $r\ge 0$, combine into a topological $\Map(\mathbb{S}^1,BG)$-colored operad $E_2^G$ that we call the \emph{operad of little $G$-bundles}.

The operad $E_2^G$ of little bundles 
is aspherical and, as our main result, we exhibit a presentation as groupoid-valued operad in terms of generators and relations (Section~\ref{seccatalg}) using so-called  \emph{parenthesized $G$-braids}. This allows us to prove in Theorem~\ref{corlittlebundlesalg} that the categorical little bundles algebras are precisely 
braided $G$-crossed categories -- a $G$-equivariant version of a monoidal braided category which is \emph{not} a braided category itself in the usual sense. This concept is due to Turaev \cite{turaevgcrossed,turaevhqft}; we, however, use a version of this notion \cite[Definition 5.4]{Coherence} allowing for more general coherence data  and omitting the requirement of rigidity (existence of duals). 
Crossed braided categories have been extensively studied e.g.\ in \cite{mueger,kirrilovg04,centerofgradedfusioncategories,enom,turaevhqft}. 
One ingredient of our proof is a recent coherence result for $G$-equivariant categories \cite{Coherence}.
In the existing literature the bookkeeping of the coherence data of a $G$-crossed braided monoidal category is done manually. Our operadic approach encodes this data in a compact way and allows to define $G$-crossed braided  algebras up to coherent homotopy beyond the category-valued case. In particular, it naturally leads to a notion of a
braided $G$-crossed
differential graded algebra or $\infty$-category.

The little bundles operad 
is closely related to the study of $G$-equivariant topological field theory \cite{turaevhqft,htv,hrt}, a flavor of topological field theory featuring bordisms equipped with principal $G$-bundles. The introduction of the bundle decoration leads to interesting phenomena such as a certain non-commutativity of the algebraic structures that can be extracted from the field theory. 

Some aspects of the little bundles operad introduced in the present paper are implicit in the literature on equivariant topological field theories
because it captures the genus zero part of surfaces decorated with $G$-bundles. This becomes manifest in the following observations:
 The extraction of $G$-crossed algebras from two-dimensional $G$-equivariant topological field theories valued in vector spaces via evaluation on the circle \cite{kaufmannorb,turaevhqft} can be understood as a statement about the path connected components of the little bundles operad (Remark~\ref{remGcrossedalgebras}). Similarly, the interpretation of the coherence diagrams for  
braided $G$-crossed categories 
in terms of low-dimensional topology as carried out in \cite{maiernikolausschweigerteq} for category-valued field theories constructed from Drinfeld doubles and in \cite{extofk} in the general case can be phrased in a unified way via the little bundles operad. 

The operadic approach is particularly well-suited for the study of $(\infty,1)$ equivariant topological field theories \cite{MuellerWoikeHH}. As a main application,
 we prove that in dimension two the value of such a theory on the circle produces a homotopy little bundles algebra
(Theorem~\ref{thmequivTFTcircle}), i.e.\
an algebra over the Boardman-Vogt resolution of the little bundles operad. 
This is a generalization of the well-known result that the value of a two-dimensional topological field theory on the circle is an $E_2$-algebra \cite{getzler,bzfn}.

	Our definition of the little bundles operad is in fact obtained as special case of a more general construction allowing both for higher dimensional disks and non-aspherical target spaces. Although the present work mainly focuses on the little bundles operad, the applications of its generalization to the study of topological field theories with non-aspherical target are indicated in Proposition~\ref{propononasphericaltarget}.

		\subparagraph{Acknowledgements and index of notation.}
	We would like to thank 
	Adrien Brochier, 
	Najib Idrissi,
	Christoph Schweigert,
	Richard Szabo and
	Nathalie Wahl
	for helpful discussions. 
We are grateful to the anonymous referee for numerous useful comments that helped improve the article.

	LM is supported by the Doctoral Training Grant ST/N509099/1 from the UK Science and Technology Facilities
	Council (STFC).
	LW is supported by the RTG 1670 ``Mathematics inspired by String theory and Quantum
	Field Theory''.

		\begin{center}
			
			\begin{tabularx}{\textwidth}{X p{0.6\textwidth} c}
				\textsc{Symbol} & \textsc{Explanation} & \textsc{Page} \\ 
				$E_n$ & little disk operad & \pageref{littlediskoperad}\\ 
				$\disk^n$ & closed $n$-dimensional standard disk & \pageref{standarddisk}\\
				$\comp(f)$ & complement of little disk embedding $f$  & \pageref{complementldo} \\ 
				$W_n(r)$ & space of pairs of embeddings of $r$ little disks and points in the complement  & \pageref{wnspace}\\
				$\mathbb{S}^{n}$ & $n$-dimensional standard sphere & \pageref{standardsphere}\\
				$W_n^T(r)$ & space of pairs of embeddings of $r$ little disks and maps to a space $T$ defined on the complement of the embedding & \pageref{spacewnt}\\
				$E_n^G$ & little bundles operad & \pageref{deflittlebundles}\\
				$\PBun_G(X)$ & groupoid of $G$-bundles over a space $X$ & \pageref{eqnmapspacebundlegrpd} \\
				$P_r$ & pure braid group on $r$ strands & \pageref{e2braidgroupeqn} \\
				$B_r$ & braid group on $r$ strands & \pageref{eqndescriptionE2} \\
				$q^{-1}[y]$ & homotopy fiber of a map $q: X \to Y$ over $y \in Y$ & \pageref{homotopyfiberpage}\\
				$\widehat{-}$ & equivalence between the groupoid of $G$-bundles over the circle and the loop groupoid of $G$ & \pageref{eqnweakinverse}\\
				$\PBr^G$ & parenthesized $G$-braids & \pageref{pbrref} \\
				$\P^G$ & operad of $G$-parentheses & \pageref{refpgauxop} \\
				$G\text{-}\Cob(n)$ & symmetric monoidal $(\infty,1)$-category of $n$-dimensional bordisms decorated with maps to $BG$ & \pageref{defZGCob}\\
				$\widehat\comp(f)$ & fattened version of $\comp(f)$ & \pageref{fatcomp}\\
				$\widehat W_n(r)$ & fattened version of $W_n(r)$ & \pageref{fatW}\\
				$\ball_\varepsilon(x_0)$ &open ball of radius $\varepsilon>0$ and center $x_0$ & \pageref{openballeps}
				\end{tabularx}
			
			\end{center}

	\section{Maps on complements of little disks}
	 For $n\ge 1$ let $E_n$ be the little $n$-disks operad.
	 Recall that the operations $E_n(r)$\label{littlediskoperad} in arity $r\ge 0$ for this topological operad are given as follows:
	Denote by $\disk^n$ the closed $n$-dimensional disk $\disk^n := \{  x \in\mathbb{R}^n \, | \, |x|^2\le 1\}$.\label{standarddisk} Then $E_n(r)$ is given by the space of all maps $f: \coprod_{k=1}^r \disk^n \to \disk^n$ such that
	\begin{itemize}
		\item the restriction $f_k : \disk^n \to \disk^n$ of $f$ to the $k$-th disk is an affine embedding, 
		i.e.\ given by a rescaling of the radius and a translation,
		\item for $1\le j<k\le r$ the interiors of the images of $f_j$ and $f_k$ do not intersect, 
		i.e.\ $\overset{\circ}{\im f_j} \cap \overset{\circ}{\im f_k} = \emptyset$.
		\end{itemize}
		For details we refer e.g.\ to \cite[Chapter~4]{FresseI}.
		
We write \label{complementldo}	$	\comp (f) := \disk^n \setminus \overset{\circ}{\im f} 	$ for the complement of the interior of the image of $f$ in the closed $n$-disk.
 This allows us to the define the subspace\label{wnspace}
	\begin{align}
	W_n(r) := \{ (f,x) \in E_n(r) \times \disk^n \, | \, x \in \comp(f)  \} \subset E_n(r) \times \disk^n\ .
	\end{align} 
The boundary $\partial \comp(f)$ of $\comp(f)$ consists of $r$ spheres $\mathbb{S}^{n-1}$\label{standardsphere} (the \emph{ingoing boundary} or \emph{ingoing spheres}) possibly wedged together and sitting inside a bigger sphere
 $\mathbb{S}^{n-1}$ (the \emph{outgoing boundary} or \emph{outgoing sphere}). 
 
The goal of this section is to define an operad $E_n^T$ whose colors are maps from $\mathbb{S}^{n-1}$ to some fixed topological space $T$ and whose operations from an $r$-tuple $\underline{\varphi}=(\varphi_1,\dots,\varphi_r)$ of maps $\mathbb{S}^{n-1}\to T$
(generally, we will use an underline to indicate tuples)
to a single map $\psi : \mathbb{S}^{n-1} \to T$ will be given by all operations $f \in E_n(r)$ together with maps $\comp (f)\to T$ whose restriction to $\partial \comp(f)$ is $(\underline{\varphi},\psi)$. 

We will introduce the operad $E_n^T$ formally in Section~\ref{secopent}. 
The definition will not be as a homotopy quotient of a braid group action (as mentioned in the introduction), but will make use of the auxiliary constructions in Section~\ref{secspacesWTaux}. The description using a homotopy quotient is then discussed in Section~\ref{sechomotopyquotient}. 

\subsection{The auxiliary spaces $W_n^T(r)$\label{secspacesWTaux}}
For a topological space $T$ we define
$W_n^T(r)$\label{spacewnt} as the set of pairs $(f,\xi)$,
where $f \in E_n(r)$ and $\xi : \comp(f) \to T$ is a continuous map. 
Projection onto the first factor yields a map
	\begin{align} 
	p : W_n^T(r) \to E_n(r)
	\end{align} of sets.
	We equip $W_n^T(r)$ with the final topology with respect to all maps of sets $g: Y \to W_n^T(r)$
	from arbitrary topological spaces $Y$ to the set $W_n^T$ such that
	\begin{enumerate} [label=(\Roman*)]
		\item $p \circ g$ is a continuous map, \label{finaltop1}
		
		\item and the natural map $W_n(r)\times_{E_n(r)}Y  \to  T $ is continuous. \label{finaltop2}
	\end{enumerate}
The further investigation of the spaces $W_n^T$ involves some elementary, but tedious point-set topology and is carried out in Appendix~\ref{appwspaces}. We only mention the crucial properties necessary in the sequel:

\begin{proposition}[Appendix, Proposition~\ref{appLemma: Serre fibration}]\label{Lemma: Serre fibration}
The map $p : W_n^T(r) \to E_n(r) $ is a Serre fibration. 
\end{proposition}

Next we define the subspace
  \begin{align}
  \partial W_n(r) := \{(f,x) \in E_n(r) \times \disk^n \, | \, x \in \partial \comp(f)\} \subset W_n(r) \ 
  \end{align} and note that there is a natural map \begin{align} E_n(r) \times \coprod^{r+1} \mathbb{S}^{n-1} \to \partial W_n(r)\label{eqncopiesofspheres}\end{align} of spaces over $E_n(r)$
 identifying the first $r$ copies of $\mathbb{S}^{n-1}$ with the ingoing boundary spheres and the last copy with the outgoing boundary spheres.
This map is generally not a homeomorphism since some of the boundary spheres might be wedged together. 
By combining the restriction of the evaluation map 
  \begin{align}
  \partial W_n(r)\times_{E_n(r)} W_n^T(r) \to T
  \end{align} (which is continuous by Lemma~\ref{lemma1})
  with \eqref{eqncopiesofspheres} we obtain a map
  \begin{align} \left(  \coprod^{r+1} \mathbb{S}^{n-1} \right) \times W_n^T(r) \to T \end{align} which, by adjunction, gives us a map
  \begin{align} q: W_n^T(r) \to\prod^{r+1} \Map (\mathbb{S}^{n-1},T) \ .    \label{therestrictionqeqn}   \end{align}

\begin{proposition}[Appendix,~Proposition~\ref{applemmaqisfibration}]\label{lemmaqisfibration} 
 The map $q: W_n^T(r) \to \prod^{r+1} \Map (\mathbb{S}^{n-1},T)$ is a Serre fibration.
\end{proposition}

\subsection{The operad $E_n^T$\label{secopent}}
From the map $q$ we can construct a topological operad colored over the set of maps $\mathbb{S}^{n-1} \to T$:
 Let $\underline{\varphi}=(\varphi_1,\dots,\varphi_r)$ be an $r$-tuple of maps $\mathbb{S}^{n-1} \to T$. Then for another map $\psi : \mathbb{S}^{n-1} \to T$ we consider the fiber $E_n^T \binom{\psi}{ \underline{\varphi}}$ of $q$ over $(\underline{\varphi},\psi)$, i.e. the pullback
    		\begin{equation}
   \begin{tikzcd}
   E_n^T \binom{\psi}{ \underline{\varphi}} \ar{rr}{} \ar[swap]{dd}{} & & W_n^T(r) \ar{dd}{q} \\
   & & \\
   \star \ar{rr}{(\underline{\varphi},\psi)}  & & \prod^{r+1} \Map (\mathbb{S}^{n-1},T) \ , 
   \end{tikzcd} \label{defEnTeqn}
   \end{equation} 
which is also a homotopy pullback since $q$ is a Serre fibration by Proposition~\ref{lemmaqisfibration}. Explicitly, $E_n^T \binom{\psi}{ \underline{\varphi}}$ consists of elements $f \in E_n(r)$ together with a map $\xi : \comp(f) \to T$ whose restriction to $\partial \comp(f)$ is given by $(\underline{\varphi},\psi)$. 
We will denote the point in $E_n^T \binom{\psi}{ \underline{\varphi}}$ formed by $f$ and $\xi$ by $\langle f,\xi\rangle$. 
  
The operad structure on $E_n^T$, which will be defined now, makes use of the operad structure of $E_n$ for which we refer to \cite[Chapter~4]{FresseI}. The operadic identity $\star \to E_n^T \binom{\varphi}{ \varphi}$ is the operadic identity in $E_n(1)$, namely the identity embedding $I : \disk^n \to \disk^n$, together with $\varphi : \comp(I) = \mathbb{S}^{n-1} \to T$. 
Moreover, the action of the symmetric group $\Sigma_r$ on $r$ letters on $E_n(r)$, for all $r\ge 0$,
turns  $E_n^T $ into a $\Map (\mathbb{S}^{n-1},T)$-colored symmetric sequence.

The operadic composition consists of maps
\begin{align}
\circ : E_n^T \binom{\psi}{ \underline{\varphi}} \times  \prod_{j=1}^r E_n^T \binom{\varphi_j}{\underline{\lambda}_j} \to E_n^T \binom{\psi}{\otimes_{j=1}^r \underline{\lambda}_j}\ ,
\end{align} where $\otimes$ denotes the juxtaposition of tuples. It sends
\begin{align}
 \left(\langle f,\xi \rangle ,  \prod_{j=1}^r \langle g_j , \mu_j \rangle      \right) \in E_n^T \binom{\psi}{ \underline{\varphi}} \times  \prod_{j=1}^r E_n^T \binom{\varphi_j}{\underline{\lambda}_j}
\end{align}
to
\begin{align}
\left\langle f \circ \underline{g}   ,    \xi \cup_{ \coprod^{r} \mathbb{S}^{n-1}   }   \underline{\mu}   \right\rangle     \in E_n^T \binom{\psi}{\otimes_{j=1}^r \underline{\lambda}_j} \ ,
\end{align} where \begin{itemize}
	\item the composition of $f$ with the $r$-tuple $\underline{g}$ of embeddings is formed via the composition in $E_n$,
	\item we use that for $1\le j \le r$ the restriction of $\mu_j : \comp(g_j) \to T$ 
	 to the last copy of $\mathbb{S}^{n-1}$ (the outer sphere) is precisely $\varphi_j$ in order to glue $\xi$ and  $\underline{\mu}$ along $r$ copies of $\mathbb{S}^{n-1}$.
	\end{itemize} 

\begin{proposition}\label{propouroperadT}
	Let $T$ be any space. With the above definitions
	$E_n^T$ is a topological operad colored over $\Map (\mathbb{S}^{n-1},T)$.
	\end{proposition}

\begin{proof}
	The only non-trivial point is the continuity of the composition maps.
	It suffices to prove that the partial compositions
	\begin{align}
	\circ_j : E_n^T \binom{ \psi  }{\varphi_1 , \dots , \varphi_j, \dots , \varphi_r}     \times E_n^T \binom{ \varphi_j  }{\underline{\lambda}}     \to E_n^T \binom{ \psi }{\varphi, \dots ,\underline{\lambda}, \dots , \varphi_r }  \label{eqnjthcomposition}
	\end{align} are continuous. 
	
	To this end, set $r' := |\underline{\lambda}|$ and 
	consider the restriction  $W_n^T(r) \to \Map(\mathbb{S}^{n-1},T)$ to the outer boundary sphere
	and the restriction $W_n^T(r') \to \Map(\mathbb{S}^{n-1},T)$ to the $j$-th ingoing boundary sphere and 
	observe that $E_n^T \binom{ \psi  }{\varphi_1 , \dots , \varphi_j, \dots , \varphi_r}     \times E_n^T \binom{ \varphi_j  }{\underline{\lambda}} $ is a subspace of the pullback
	$W_n^T(r) \times_{  \Map(\mathbb{S}^{n-1},T)  } W_n^T(r')$ 
	such that \eqref{eqnjthcomposition} is the restriction of the map
	\begin{align}
	\widehat{\circ}_j : W_n^T(r) \times_{  \Map(\mathbb{S}^{n-1},T)  } W_n^T(r') \to W_n^T (r+r'-1), \quad\left(   \langle f,\xi\rangle  ,  \langle f',\xi'  \rangle \right) \mapsto \langle  f \circ f'  ,  \xi \cup_{\mathbb{S}^{n-1}}^j \xi' \rangle  .  \label{eqnjthcompositionalt}
	\end{align} 
	Here $f \circ_j f'$ is the (partial) operadic composition in $E_n$ and $\xi \cup_{\mathbb{S}^    {n-1} }^j \xi'$ is the map obtained from gluing $\xi$ and $\xi'$ along the $j$-th sphere $\mathbb{S}^{n-1}$ in the domain of definition of $\xi$. The statement follows now from Lemma~\ref{lemmagluingmaps}
	asserting that $\widehat{\circ}_j$ is continuous.
		\end{proof}

	\begin{remark}\label{remEnTsigmacof}
An operad is \emph{$\Sigma$-cofibrant} if the underlying symmetric sequence is cofibrant in the projective model
structure. This property is important for the homotopy theory of algebras over this operad and needed later in   
Section~\ref{secgenrel}.
	The model for $E_n$ used in this article is $\Sigma$-cofibrant \cite[page~140]{FresseI}. By the same arguments, $E_n^T$ is $\Sigma$-cofibrant.
	\end{remark}

\section{The operad $E_2^G$ of little $G$-bundles}
Let us specialize the operad from Proposition~\ref{propouroperadT} to aspherical spaces $T$ to obtain what we will refer to as \emph{little bundles operad}.

Recall that a space or simplicial set $T$ called \emph{aspherical} if $\pi_k(T)=0$ for $k\ge 2$ and all choices of basepoints. Accordingly, we call an operad in spaces or simplicial sets \emph{aspherical} if all its components are aspherical. 
	
If $T$ is an aspherical space (which we will assume to be connected without loss of generality), then, up to equivalence, $T$ is the classifying space of its fundamental group $G$.
 Therefore, we set $E_n^G := E_n^{BG}$ for any (discrete) group $G$.

For a manifold $X$ (possibly with boundary) the mapping space $\Map(X,BG)$ is the nerve $B\PBun_G(X)$ 
of the groupoid of principal $G$-bundles over $X$, i.e.\
\begin{align} \Map(X,BG) \simeq B\PBun_G(X) \ , \label{eqnmapspacebundlegrpd}\end{align} 
which implies 
\begin{align} \Pi \Map(X,BG) \simeq \PBun_G(X) \ , \label{eqnmapspacebundlegrpd2}
\end{align}
where $\Pi$ denotes the fundamental groupoid functor.
A proof of these well-known facts is given in \cite[Lemma~2.8]{extofk}.

 In particular, $\Map(X,BG)$ is aspherical again with
\begin{align}
\pi_0(\Map(X,BG)) &\cong \pi_0( \PBun_G(X)  ) \ , \label{eqnmapspace1}\\ 
\pi_1 ( \Map(X,BG) , \varphi   ) &\cong \Aut( \varphi^* EG ) \ , \label{eqnmapspace2}
\end{align} 
where $\varphi^* EG$ is the pullback of the universal $G$-bundle $EG \to BG$ along a map $\varphi : X \to BG$ and where we denote by $\Aut(P)$ the group of automorphisms of a $G$-bundle $P$ (the group of \emph{gauge transformations}). 

Recall that if $X$ is connected, we find
\begin{align}
\PBun_G(X) \simeq \Hom(\pi_1(X),G) //G \label{holonomyclassification}
\end{align} by the holonomy classification of $G$-bundles, see e.g.\
\cite[Theorem~13.2]{taubes},
i.e.\ after choice of a basepoint in $X$,
the bundle groupoid $\PBun_G(X)$ is equivalent to the action groupoid associated to the action of $G$ by conjugation on the set of group morphisms $\pi_1(X)\to G$. 

Note that for $n> 2$ the operad $E_n^G$ is not really interesting  since 
all $G$-bundles over $\mathbb{S}^{n-1}$ for $n>2$ are trivializable. The case relevant to us is $n=2$: 
\begin{definition}\label{deflittlebundles}
We call the topological operad $E_2^G$ the \emph{little bundles operad}.
\end{definition}

In the remaining subsections of this section, we will show that $E_2^G$ is aspherical (Proposition~\ref{propolittlebundlesaspherical})
and will explicitly describe its components as action groupoids (Proposition~\ref{propoE2Gactiongroupoid}). 

 	\subsection{The space $W_2^G$ as a Hurwitz space\label{sechomotopyquotient}}
	In a first step we investigate the spaces $W_2^G(r) := W_2^{BG}(r)$ for a group $G$.
	
For this recall from \cite[Chapter~5]{FresseI} that $E_2(r)$ is the classifying space of the pure braid group $P_r$ on $r$ strands, i.e.
	\begin{align}
	E_2(r) \simeq BP_r \ . \label{e2braidgroupeqn}
	\end{align}
Alternatively (and for our applications more conveniently), we can describe the fundamental groupoid $\Pi E_2(r)$ as the action groupoid 
\begin{align}  
	\Pi E_2(r) \simeq \Sigma_r // B_r\ , \label{eqndescriptionE2}\end{align} where the braid group $B_r$ acts on $\Sigma_r$ by $c.\sigma := \pi(c) \sigma$ for $c\in B_r$ and $\sigma \in \Sigma_r$, i.e.\ via the projection $\pi : B_r \to \Sigma_r$ fitting into the short exact sequence
	\begin{align}
	0 \xrightarrow{\ \phantom{\pi}\   } P_r \xrightarrow{\ \phantom{\pi}\   }B_r \xrightarrow{\ {\pi}\   }\Sigma_r \xrightarrow{\ \phantom{\pi}\   }0 \ . 
	\end{align}
	If we consider
	the long exact sequence of homotopy groups for the Serre fibration from Proposition~\ref{Lemma: Serre fibration} whose fibers we computed in Lemma~\ref{lemmafibercompactopen} and take 
	\eqref{eqnmapspace1}, \eqref{eqnmapspace2} and \eqref{e2braidgroupeqn}
	into account, we arrive at:

\begin{lemma}\label{lemmahomotopyw2}
The space $W_2^G(r)$ is aspherical, and for $f \in E_2(r)$ and $\varphi \in \Map (\comp(f),BG)$ there is an exact sequence
\begin{align} 0 \to \Aut(\varphi^* EG) \to \pi_1( W_2^G(r),  \langle f,\varphi\rangle  ) \to P_r \to \pi_0( \PBun_G(\comp(f)) ) \to \pi_0(W_2^G(r)) \to 0 \ .  \end{align}
\end{lemma}

We will denote the homotopy fiber of a map $q: X \to Y$ over $y\in Y$ by $q^{-1}[y]$. If $X$ and $Y$ are aspherical, we can make the following elementary observation: \label{homotopyfiberpage}

\begin{lemma}\label{lemmahtpfiberPi}
Let $ q : X \to Y$ be a map between aspherical spaces, then for $y\in Y$ the natural map
		\begin{align}
		\Pi (	q^{-1}[y]) \to      \Pi(q)^{-1}[y] 
		\end{align} 
from the fundamental groupoid of the homotopy fiber $q^{-1}[y]$ to the homotopy fiber $\Pi(q)^{-1}[y]$ of $\Pi(q) : \Pi(X) \to \Pi(Y)$ over $y \in Y$ is an equivalence. 
\end{lemma}
	
	\begin{proof}
		Since $\Pi$ sends Serre fibrations to categorical fibrations, it suffices to prove the statement for a Serre fibration and the actual fibers instead of homotopy fibers.
		
		The spaces involved are aspherical, hence we only need to prove that the map
		\begin{align}
		\pi_0 (q^{-1}(y)) \to \pi_0    ( \Pi(q)^{-1}(y) )        \label{surjectiveonpi0}
		\end{align} is bijective and that the map
		\begin{align}
		\pi_1(q^{-1}(y),x) \to \pi_1(\Pi(q)^{-1}(y) , x) \label{eqnisoonpi1}
		\end{align} is a group isomorphism for all $x\in X$ such that $q(x)=y$. 
		
		Surjectivity of \eqref{surjectiveonpi0} follows from the definitions.
			To see injectivity of \eqref{surjectiveonpi0}, let $x$ and $x'$ be points such that $q(x)=y=q(x')$ and $x\cong x'$ in $\Pi(q)^{-1}(y)$ via a morphism $g : x \to x'$. Then $g$ can be represented as a path in $X$ from $x$ to $x'$ such that $q(g)$ is homotopic relative boundary to the constant path at $y$. This homotopy has a $q$-lift starting at $g$. The endpoint is a path from $x$ to $x'$ in $q^{-1}(y)$ proving that $x$ and $x'$ lie in the same component of $q^{-1}(f)$.
			   
		The fact that \eqref{eqnisoonpi1} is an isomorphism follows by comparing the exact sequences that both  $q^{-1}(y)$ and $ \Pi(q)^{-1}(y) $ give rise to.
	\end{proof}
	
The following proposition will be the key for understanding the auxiliary spaces $W_2^G(r)$. It provides a link to a certain flavor of Hurwitz spaces, see also Remark~\ref{remhurwitz} below. 
	
\begin{proposition}    \label{propohomotopyw3}
There is an equivalence
\begin{align} \label{thehocolimeqn} W_2^G(r) \simeq \underset{f \in \Pi E_2(r)}{\hocolim}\, \Map( \comp(f),BG  ) \ .     \end{align}
\end{proposition}
 Here by an equivalence we mean that there exists a zigzag of equivalences, i.e.\ the objects are isomorphic in the homotopy category. 
	
	\begin{proof}
	 Since $p : W_2^G(r) \to E_2(r)$ is a Serre fibration (Lemma~\ref{Lemma: Serre fibration}), $\Pi (p) : \Pi W_2^G(r) \to \Pi E_2(r)$ is a categorical fibration. From this we easily deduce that it is also a category fibered in groupoids in the sense of \cite{delignemumford}.
	Corresponding to this category fibered in groupoids
	 we have by \cite[Section~3.3]{hollander} 
	a (pseudo-)functor $X: \left(  \Pi E_2(r)\right)^\opp \to \Grpd$ 
	(we can also see this as a $\Pi E_2(r)$-shaped diagram since a groupoid is equivalent to its opposite)
	such that for $f\in E_2(r)$ the groupoid $X(f)$ is equivalent to the fiber of $\Pi W_2^G(r) \to \Pi E_2(r)$ over $f$. This fiber is equivalent to $\Pi p^{-1}(f)$ by Lemma~\ref{lemmahtpfiberPi} and, finally, to $ \Pi \Map( \comp(f),BG  )$ by Lemma~\ref{lemmafibercompactopen}. If we denote by $\int$ the Grothendieck construction, we conclude from \cite[Theorem~3.12]{hollander} that there is a canonical fiberwise equivalence
			\begin{align}\int X \to \Pi W_2^G(r)  \label{eqncounitgrothendieck}\end{align} of groupoids over $\Pi E_2(r)$. It is then straightforward to verify that this is also an equivalence of groupoids.
		
		By Thomason's Theorem \cite[Theorem~1.2]{thomason}  we obtain a canonical equivalence
		\begin{align}
		\underset{f \in \Pi E_2(r)}{\hocolim}\, BX(f) \xrightarrow{\    \simeq  \    }       B  \int X \ . 
		\end{align}	
		Combining this with the equivalence
		\eqref{eqncounitgrothendieck}
		yields the assertion if we additionally take into account that  $\Map(\comp(f),BG)$ and $W_2^G(r)$
		are aspherical by  \eqref{eqnmapspacebundlegrpd} and Lemma~\ref{lemmahomotopyw2}, respectively.
	\end{proof}

Since $\comp(f)$ is equivalent to a wedge $\bigvee_{j=1}^r \mathbb{S}^1$ of $r$ circles, we conclude from \eqref{eqnmapspacebundlegrpd} and \eqref{holonomyclassification}
	\begin{align} \Map( \comp(f),BG  ) \simeq B \left(  \Hom (\mathbb{Z}^{*r} , G) // G \right) \simeq B (G^{\times r} // G) \ ,  \label{eqnidentifikationbungrpd}   \end{align}
	where $\mathbb{Z}^{*r}$ is the free group on $r$ generators.
	
\begin{lemma}\label{lemdiagramforw2}
Under the identifications \eqref{eqndescriptionE2}  and \eqref{eqnidentifikationbungrpd}, the diagram from $\Pi E_2(r)$ to spaces underlying the homotopy colimit \eqref{thehocolimeqn} is point-wise the nerve of the diagram
		\begin{align}
		\Sigma_r // B_r \to \Grpd 
		\end{align} 
sending $\sigma \in \Sigma_r$ to $G^{\times r} // G$. 
The generator $c_{j,j+1}\in B_r$ braiding strand $j$ and $j+1$ acts as the automorphism
		\begin{align}
		G^{\times r} // G \to G^{\times r} // G, \quad (g_1,\dots,g_j,g_{j+1},\dots,g_r) \mapsto (g_1,\dots,g_jg_{j+1}g^{-1}_j,g_{j},\dots,g_r) \ .    \label{eqnactionofbraidgroup}
		\end{align}
		\end{lemma}

		\begin{proof}
		We only have to observe that the transformation of holonomies under the braid group action is given by the formula \eqref{eqnactionofbraidgroup} (sometimes called \emph{Hurwitz formula}). For a detailed proof of this fact (given without loss of generality for two embedded disks) we refer to e.g.\ \cite[Lemma~3.25]{maiernikolausschweigerteq}.
		\end{proof}

By Proposition~\ref{propohomotopyw3}, $W_2^G(r)$ is the homotopy colimit of the nerve of the diagram presented in Lemma~\ref{lemdiagramforw2}. 
For later purposes, we need to describe $W_2^G(r)$ explicitly as a groupoid. 
If we combine Lemma~\ref{lemdiagramforw2} with the homotopy colimit formula provided in Lemma~\ref{lemhocolimexp} in the Appendix, we obtain:

\begin{lemma}\label{corpresentationw2g}
The groupoid $\Pi W_2^G(r)$ is equivalent to the groupoid with objects 
$\Sigma_r\times  G^{\times r}$ and pairs $(c,h)\in B_r \times G$ as morphisms
$(\sigma,g_1,\dots,g_r ) \to (\pi(c)\sigma , c. (hg_1h^{-1} , \dots, hg_rh^{-1}))$, where the action of the braid group on tuples of group elements is given by \eqref{eqnactionofbraidgroup}.
\end{lemma}

\begin{remark}\label{remhurwitz}
The homotopy quotient of the space of $G$-bundles over a punctured plane by the braid group action (or its description in terms of holonomies) first appeared in \cite{clebsch,hurwitz} and is called a \emph{Hurwitz space}, see \cite{evw} for an overview. 
\end{remark}

\subsection{Groupoid description of $E_2^G$}
Our investigation of $W_2^G$ is the key to the computation of the homotopy groups of the little bundles operad $E_2^G$. 
Using the long exact sequence for the Serre fibration $q : W_2^G(r) \to \prod^{r+1} \Map (\mathbb{S}^{1},BG)$ 
from Proposition~\ref{lemmaqisfibration}
combined with 
Lemma~\ref{lemmahomotopyw2}
we obtain:  
	\begin{proposition}\label{propolittlebundlesaspherical}
	For any group $G$, the operad $E_2^G$ is aspherical.
\end{proposition}

Therefore, it suffices to compute the groupoid-valued operad $\Pi E_2^G$.
To this end, recall from \eqref{eqnmapspacebundlegrpd2} that the groupoid $\Pi \Map(\mathbb{S}^1,BG)$ is canonically equivalent to the groupoid of $G$-bundles over $\mathbb{S}^1$, hence for any fixed choice of basepoint we obtain 
an equivalence
\begin{align}
\Pi \Map(\mathbb{S}^1,BG) \xrightarrow{\   \simeq \ } G//G \ .  \label{equivtoactiongroupoid}
\end{align}
In the sequel, we choose a weak inverse
\begin{align}
\widehat{-} : G//G \xrightarrow{\   \simeq \ } \Pi \Map(\mathbb{S}^1,BG) \ . \label{eqnweakinverse}
\end{align}
We make our choices such that the unit element
$e$ of the group $G$ is mapped to the constant loop at the base point, and such that all loops in the image map $(0,1)\in \mathbb{S}^1\subset \mathbb{R}^2$ 
to the base point of $BG$. 
The object function
$\widehat{-}:  G \to  \Map(\mathbb{S}^1,BG)$ can be used to pull back $E_2^G$ to a $G$-colored operad whose components are given as follows:

\begin{proposition}\label{propoE2Gactiongroupoid}
For $\underline{g}\in G^r$ and $h\in G$	
the groupoid $\Pi E_2^G \binom{\widehat{h}}{ \widehat{\underline{g}}}$ is equivalent to the action groupoid of the $B_r$-action specified in Lemma~\ref{corpresentationw2g} on the subset 
	\begin{align}
\Sigma_r \times_h G^r := \left\{  	(\sigma,\underline{b})\in \Sigma_r \times G^r \, \left| \,	\prod_{j=1}^r b_{\sigma(j)}       g_j b_{\sigma(j)}^{-1}  = h\right. \right\} \subset \Sigma_r \times G^r \ .
	\end{align}
	\end{proposition}
	
	\begin{proof}
	By Lemma~\ref{lemmahtpfiberPi}, $\Pi E_2^G \binom{h}{ \underline{g}}$ is equivalent to the homotopy fiber of \begin{align}\label{maphtpfiberinteqn} \Pi W_2^G(r) \to \Pi \prod^{r+1} \Map (\mathbb{S}^{1},BG)\simeq (G//G)^{r+1}       \end{align} over $(\underline{g},h)$. 
	Using the presentation of 
	$\Pi W_2^G(r)$ given in
	Lemma~\ref{corpresentationw2g}, the functor \eqref{maphtpfiberinteqn} sends $(\sigma,a_1,\dots,a_r)$ to 
	$(a_1,\dots,a_r, a_{\sigma(1)} \dots a_{\sigma(r)})$.
	Therefore, the homotopy fiber of \eqref{maphtpfiberinteqn} over $(\underline{g},h)$ consists of all
	\begin{align} (\sigma,\underline{a}=(a_1,\dots,a_r)) & \in \Sigma_r \times G^r, \quad \underline{b}= (b_1,\dots,b_{r+1})\in G^{r+1} \\  \text{such that}\quad b_j a_j b_j^{-1} &= g_j, \quad 1\le j\le r, \quad b_{r+1} a_{\sigma(1)} \dots a_{\sigma(r)} b_{r+1}^{-1} = h \ .
	\end{align} From Lemma~\ref{corpresentationw2g} it follows that, up to equivalence, we can concentrate on the full subgroupoid of the homotopy fiber spanned by those objects satisfying $b_{r+1}=1$; 
	and a morphism $(\sigma,\underline{a},\underline{b}) \to (\sigma', \underline{a'},\underline{b'})$ in that full subgroupoid, i.e.\ with $b_{r+1}=b_{r+1}'=1$, is just an element of $B_r$.
	Of course, for an object $(\sigma,\underline{a},\underline{b})$ the tuple $\underline{a}$ is redundant because $a_j = b_j^{-1} g_j b_j$. Also, we may work with the tuple $\underline{b}^{-1}=(b_1^{-1},\dots,b_r^{-1})$ instead of $\underline{b}$. 
	\end{proof}

\begin{remark}
	The above statement just gives the components of $\Pi E_2^G$, but does not give a description as an operad. 
	The latter problem will be addressed in Section~\ref{seccatalg}.
	\end{remark}

In the sequel, we will need a lifting result for the functor $	\Pi E_2^G \binom{\psi}{ \underline{\varphi}} \to \Pi E_2(r) $.

\begin{proposition}\label{Prop: Unique lifts}
	For $\underline{\varphi} \in \prod^r \Map (\mathbb{S}^{1},BG)$ and $\psi \in \Map (\mathbb{S}^{1},BG)$ the forgetful functor
	\begin{align}
	\Pi E_2^G \binom{\psi}{ \underline{\varphi}} \to \Pi E_2(r)     \label{e2ge2forgetfulfibeqn}
	\end{align} 
	admits
	lifts of the form
	 \begin{equation}
		\begin{tikzcd}
		0  \ar{rr}{  x_0   } \ar[swap]{dd}{} & & 	\Pi E_2^G \binom{\psi}{ \underline{\varphi}}   \ar{dd}{} \\
		& & \\
	\text{$[1]$}  \ar{rr}{   g  } \ar[dashrightarrow]{rruu}{} & & \Pi E_2(r)
		\end{tikzcd} 
		\end{equation}
		as long as the end and starting point of $g$ 
		are points in $E_2(r)$ 
		whose little disks have non-intersecting boundaries.
		\end{proposition}

However in general, \eqref{e2ge2forgetfulfibeqn} is not a fibration as the following counterexample shows:
Consider for $g,h\in G$ the diagram 
\begin{equation}
\begin{tikzcd}
0 \ar{r}{\widehat{h}} \ar[d] & \Pi E_2^G \binom{\widehat{hgh^{-1}}}{ \widehat g} \ar[d] \\ 
{[1] } \ar{r}{L} & \Pi E_2(1) \ ,
\end{tikzcd} 
\end{equation}	
where $L$ is the homotopy sketched in Figure~\ref{Fig:Delta} on page \pageref{Fig:Delta} and
$\widehat{h}$ is seen as a point in $\Pi E_2^G \binom{\widehat{hgh^{-1}}}{ \widehat g} $ by placing the homotopy corresponding to $h$ on the complement of $L_0$.
Clearly, 
for this square, there is no lift  to $ \Pi E_2^G \binom{\widehat{hgh^{-1}}}{ \widehat g}$ whenever $h\neq e$.

\begin{proof}[Proof of Proposition~\ref{Prop: Unique lifts}]
We start by showing that
for $\underline{\varphi} \in \prod^r \Map (\mathbb{S}^{n-1},BG)$ and $\psi \in \Map (\mathbb{S}^{n-1},T)$ the composition $E_n^T \binom{\psi}{ \underline{\varphi}} \to W_n^T(r) \to E_n(r)$ admits lifts for paths $I \longrightarrow E_n(r)$ 	whose little disks have non-intersecting boundaries.
	Indeed, the needed lifts 
	  \begin{equation}
	\begin{tikzcd}
	 0  \ar{rr}{} \ar[swap]{dd}{} & & E_2^G \binom{\psi}{ \underline{\varphi}} \ar{dd}{} \\
	  	& & \\
	   I  \ar{rr}{} \ar[dashrightarrow]{rruu}{ } & & E_2(r)
	\end{tikzcd} 
	\end{equation} can be constructed thanks to Lemma~\ref{Lemma: Serre fibration} if we allow the lift to take values in $W_2^G(r)$ without making sure that we hit the correct fiber. 
	The parameter $t\in I$ will then describe a path in $W_2^G (r)$
	whose restriction to the boundary circles will describe homotopies of $\underline{\varphi}$ and $\psi$. In order to remain in the fiber $ E_n^T \binom{\psi}{ \underline{\varphi}}$,
	 these restrictions would have to be constant. We can easily achieve that by fixing small non-intersecting
	 collars around the boundary circles 
	 (this is possible since we assumed that the boundaries of the disks do not intersect)
	 on which we place the inverses of the homotopies of $\underline{\varphi}$ and $\psi$ mentioned above
	 (this strategy is explained in more detail in the proof of Proposition~\ref{applemmaqisfibration} in the Appendix). 
	 
More generally, a path in $\Pi E_2^G$, for which the start and end point do not contain disks with intersecting boundaries,
admits a representative in $E_2^G$ whose little disks do not touch (by rescaling in the interior). 
\end{proof}

\begin{lemma}\label{lemmacoveringlift}
The forgetful functor $\left( \Sigma_r \times_h G^r  \right)  // B_r \to \Sigma_r // B_r$ admits a unique solution to the lifting problem
	\begin{equation}
	\begin{tikzcd}
	0  \ar{rr}{(\sigma, \underline{b})} \ar[swap]{dd}{} & &  \left( \Sigma_r \times_h G^r  \right)  // B_r   \ar{dd}{} \\
	& & \\
	\text{$[1]$}  \ar{rr}{c : \sigma \to \pi(c) \sigma } \ar[dashrightarrow]{rruu}{\exists!} & & \Sigma_r // B_r \ . 
	\end{tikzcd} 
	\end{equation}
\end{lemma}

\begin{proof}The unique lift is $c: (\sigma, \underline{b}) \to c.(\sigma,\underline{b})$.\end{proof}

\begin{remark}[Uniqueness of the lifts in Proposition~\ref{Prop: Unique lifts}]\label{remProp: Unique lifts}
For two lifts $\widetilde g : x_0 \to x_1$ and $\widetilde g' : x_0 \to x_1'$ of $g: [1] \to \Pi E_2(r)$ under 
\eqref{e2ge2forgetfulfibeqn},
there is a unique morphism $h = \widetilde g' \widetilde g^{-1} : x_1 \to x_1'$ such that $h\widetilde g = \widetilde g'$ and $h=\id_{x_1}$ whenever $x_1=x_1'$, i.e.\ the lift is completely determined by its start and end point.	 
For the proof of this uniqueness statement, we consider the commutative diagram
\begin{equation}
 \begin{tikzcd}
 0  \ar{rr}{  x_0   } \ar[swap]{dd}{} & & 	\Pi E_2^G \binom{\psi}{ \underline{\varphi}}   \ar{dd}{} \ar{rr}{\simeq} && \left( \Sigma_r \times_h G^r  \right)  // B_r  \ar{dd}{}  \\
 & & \\
 \text{$[1]$}  \ar{rr}{   g  }  & & \Pi E_2(r) \ar{rr}{\simeq} && \Sigma_r // B_r\ ,
 \end{tikzcd} 
 \end{equation}
where $r= |\underline{\varphi}|$. Moreover, we have used a holonomy description of $(\underline{\varphi},\psi)$ to apply Proposition~\ref{propoE2Gactiongroupoid}.
	 We deduce from Lemma~\ref{lemmacoveringlift} that the images of two lifts $\widetilde g : x_0 \to x_1$ and $\widetilde g' : x_0 \to x_1'$ under $	\Pi E_2^G \binom{\psi}{ \underline{\varphi}} \to \left( \Sigma_r \times_h G^r  \right)  // B_r $ agree. 
This implies $\Phi(h)=\id_{\Phi(x')}$ and hence $h=\id_{x_1}$ whenever $x_1=x_1'$. 

Furthermore, the set of allowed endpoints is the preimage of the endpoint of the lift from Lemma \ref{lemmacoveringlift} under the equivalence $\Pi E_2^G \binom{\psi}{ \underline{\varphi}} \to \left( \Sigma_r \times_h G^r  \right)  // B_r $ intersected with the preimage of the endpoint of $g$ under the projection $\Pi E_2^G \binom{\psi}{ \underline{\varphi}}\to \Pi E_2(r)$.  
\end{remark}

\section{Categorical algebras over the little bundle operad\label{seccatalg}}
Since any aspherical operad can be seen as an operad in groupoids, it is natural to consider its categorical algebras. For the little disks operad this leads to braided monoidal categories, see \cite[Chapter 5 and 6]{FresseI} for a detailed discussion that will also be briefly summarized below. For the little bundles operad, as we prove in this section, this leads to braided $G$-crossed categories.
This type of category, which is of great importance in equivariant representation theory \cite{centerofgradedfusioncategories,enom} and topological field theory \cite{htv,hrt,extofk}, is based on work by Turaev  \cite{turaevgcrossed,turaevhqft}. Various flavors of the notion exist \cite{maiernikolausschweigerteq,Coherence} differing, for example, by the type of coherence data considered. 
 We will give the precise version relevant to the present article below.

\subsection{Groupoid-valued operads in terms of generators and relations}
In this subsection we recall the definition of an operad in terms of generators and relations; we refer to \cite{FresseI,Yau} for details.

 For a fixed non-empty set $\mathfrak{C}$ of colors we denote by $U: \Op(\mathcal{M}) \to \SymS (\mathcal{M})$ the forgetful functor from the category of $\mathfrak{C}$-colored operads valued in a bicomplete closed symmetric monoidal category $\mathcal{M}$ to the category of symmetric sequences in $\mathcal{M}$.
 This functor admits a left adjoint $F: 
\SymS (\mathcal{M})\longrightarrow \Op(\mathcal{M})$, the free operad functor.

The free operad functor and the cocompleteness of the category of operads can be used to define
 an operad via generators and
relations: Fix a collection of generators $G\in \SymS(\mathcal{M})$ and 
relations $R\in \SymS(\mathcal{M})$ together with two morphisms 
$r_1,r_2 : R \longrightarrow UF(G)$.  Via the adjunction 
$F\dashv U$, this defines two morphisms $F(R) \substack{  \to \\ \to  } F(G)$. The operad 
generated by $G$ and $R$ is the coequalizer of the parallel pair $F(R)  \substack{  \to \\ \to  }  F(G)$.

In the case that $\mathcal{M}$ is the category of groupoids, $\mathcal{M}=\Grpd$,
 we draw an object $g$ of 
the groupoid $G\binom{t}{ (c_1,\dots, c_n)}$ as a planar graph 
with one vertex 
labeled by $g$, $n$ ingoing legs labeled by $c_1, \dots, c_n$ and one
outgoing edge labeled by $t$. For example,
 we depict an object 
$g\in G\binom{t}{(c_1,c_2)}$ as 
\begin{flalign}\label{eqn:genpics}
\begin{tikzpicture}[cir/.style={circle,draw=black,inner sep=0pt,minimum size=2mm},
        poin/.style={circle, inner sep=0pt,minimum size=0mm}]
\node[poin] (Mout) [label=above:{\small $t$}] at (8,1) {};
\node[poin] (Min1) [label=below:{\small $c_1$}] at (7.7,0) {};
\node[poin] (Min2) [label=below:{\small $c_2$}] at (8.3,0) {};
\node[poin] (V)  [label=left:{\small $g$}] at (8,0.5) {};
\draw[thick] (Min1) -- (V);
\draw[thick] (Min2) -- (V);
\draw[thick] (V) -- (Mout);
\end{tikzpicture}. 
\end{flalign} 
We will draw morphisms as dotted lines between trees. For example, we
depict a morphism $f: g\longrightarrow g' \in G\binom{t}{(c_1,c_2)}$     as    
\begin{flalign}\label{eqn:genpics}
\begin{tikzpicture}[cir/.style={circle,draw=black,inner sep=0pt,minimum size=2mm},
        poin/.style={circle, inner sep=0pt,minimum size=0mm}]
\node[poin] (Mout) [label=above:{\small $t$}] at (7,1) {};
\node[poin] (Min1) [label=below:{\small $c_1$}] at (6.7,0) {};
\node[poin] (Min2) [label=below:{\small $c_2$}] at (7.3,0) {};
\node[poin] (V)  [label=left:{\small $g$}] at (7,0.5) {};
\node[poin] (Mout') [label=above:{\small $t$}] at (9,1) {};
\node[poin] (Min1') [label=below:{\small $c_1$}] at (8.7,0) {};
\node[poin] (Min2') [label=below:{\small $c_2$}] at (9.3,0) {};
\node[poin] (V')  [label=right:{\small $g'$}] at (9,0.5) {};
\node[poin] (A) [label=above:{\small $f$}] at (8,0.5) {};
\draw[thick] (Min1) -- (V);
\draw[thick] (Min2) -- (V);
\draw[thick] (V) -- (Mout);
\draw[thick] (Min1') -- (V');
\draw[thick] (Min2') -- (V');
\draw[thick] (V') -- (Mout');
\draw[dotted,->](V) -- (A) ;
\draw[dotted] (A) -- (V') ;
\end{tikzpicture}.
\end{flalign}    
Furthermore, when we draw a list of generators, we only draw `elementary'
generators and add elements corresponding to the action of the
permutation group and the composition of morphisms. Put more formally,
we 
only specify the groupoid as a directed graph,
 take the free groupoid 
generated by this graph and add elements corresponding to the action 
of the permutation group freely. 
Note that this automatically adds inverses for 
every morphism.

To simplify the notation later on, we draw diagrams like 
\begin{flalign}
\begin{tikzpicture}[cir/.style={circle,draw=black,inner sep=0pt,minimum size=2mm},
        poin/.style={circle, inner sep=0pt,minimum size=0mm}]
\node[poin] (Mout) [label=above:{\small $t$}] at (7.3,1.5) {};
\node[poin] (Min1) [label=below:{\small $c_1$}] at (6.7,0) {};
\node[poin] (Min2) [label=below:{\small $c_2$}] at (7.3,0) {};
\node[poin] (Min3) [label=below:{\small $c_3$}] at (7.9,0) {};
\node[poin] (V)  [label=left:{\small $g_1$}] at (7,0.5) {};
\node[poin] (V2)  [label=left:{\small $g_2$}] at (7.3,1) {};
\node[poin] (Mout') [label=above:{\small $t$}] at (9.3,1.5) {};
\node[poin] (Min1') [label=below:{\small $c_1$}] at (8.7,0) {};
\node[poin] (Min2') [label=below:{\small $c_2$}] at (9.3,0) {};
\node[poin] (Min3') [label=below:{\small $c_3$}] at (9.9,0) {};
\node[poin] (V')  [label=right:{\small $g_1'$}] at (9.6,0.5) {};
\node[poin] (V2')  [label=right:{\small $g_2'$}] at (9.3,1) {};
\node[poin] (A) [label=above:{\small $\alpha$}] at (8.3,0.5) {};
\draw[thick] (Min1) -- (V);
\draw[thick] (Min2) -- (V);
\draw[thick] (Min3) -- (V2);
\draw[thick] (V2) -- (Mout);
\draw[thick] (V) -- (V2);
\draw[thick] (Min1') -- (V2');
\draw[thick] (Min2') -- (V');
\draw[thick] (Min3') -- (V');
\draw[thick] (V2') -- (Mout');
\draw[thick] (V') -- (V2');
\draw[dotted,->](7.7,0.5) -- (A) ;
\draw[dotted] (A) -- (8.8,0.5) ;
\end{tikzpicture}
\end{flalign} 
with generators $g_1,g_2,g_1',g_2'$ 
to describe the following:
We formally add objects $A$ and $B$ in $G\binom{t}{(c_1,c_2,c_3)}$, a morphism
$\alpha : A \longrightarrow B$ between them and afterwards impose the relation
\begin{flalign}
\begin{tikzpicture}[cir/.style={circle,draw=black,inner sep=0pt,minimum size=2mm},
        poin/.style={circle, inner sep=0pt,minimum size=0mm}]
\node[poin] (Mout) [label=above:{\small $t$}] at (5.3,1.5) {};
\node[poin] (Min1) [label=below:{\small $c_1$}] at (4.7,0) {};
\node[poin] (Min2) [label=below:{\small $c_2$}] at (5.3,0) {};
\node[poin] (Min3) [label=below:{\small $c_3$}] at (5.9,0) {};
\node[poin] (V)  [label=left:{\small $g_1$}] at (5,0.5) {};
\node[poin] (V2)  [label=left:{\small $g_2$}] at (5.3,1) {};
\node[poin] (Mout') [label=above:{\small $t$}] at (9.3,1.5) {};
\node[poin] (Min1') [label=below:{\small $c_1$}] at (8.7,0) {};
\node[poin] (Min2') [label=below:{\small $c_2$}] at (9.3,0) {};
\node[poin] (Min3') [label=below:{\small $c_3$}] at (9.9,0) {};
\node[poin] (V')  [label=right:{\small $g_1'$}] at (9.6,0.5) {};
\node[poin] (V2')  [label=right:{\small $g_2'$}] at (9.3,1) {};
\node[poin] (A) [label=above:{$=$}] at (8.3,0.5) {};
\node[poin] (A) [label=above:{$B$}] at (7.5,0.5) {};
\node[poin] (A) [label=above:{$=$}] at (4.3,0.5) {};
\node[poin] (A) [label=above:{$A$}] at (3.5,0.5) {};
\node[poin] (.) [label=right:{,}] at (11,0) {};
\draw[thick] (Min1) -- (V);
\draw[thick] (Min2) -- (V);
\draw[thick] (Min3) -- (V2);
\draw[thick] (V2) -- (Mout);
\draw[thick] (V) -- (V2);
\draw[thick] (Min1') -- (V2');
\draw[thick] (Min2') -- (V');
\draw[thick] (Min3') -- (V');
\draw[thick] (V2') -- (Mout');
\draw[thick] (V') -- (V2');
\end{tikzpicture} 
\end{flalign} 
where the 
respective right hand sides of the equations describe the operadic composition of generators of $G$ in the operad $FG$ via trees.

Recall that the universal property of the coequalizer and the universal property of the free operad 
allow us to describe algebras over an operad defined in terms of generators and relations very concretely:

\begin{proposition}\label{Algebras with generators and relations}
An algebra $A$ in the category $\Cat$ of small categories 
over a $\mathfrak{C}$-colored operad in $\Grpd$ described by generators 
$G$ and relations $R$ consists of 
\begin{itemize}
\item a category $A_c$ for every $c \in \mathfrak{C}$, 
\item a functor
$A_g : A_{c_1}\times \dots \times  A_{c_n}\longrightarrow A_t$ for every generating object $g\in G\binom{t}{\underline{c}}$,
\item a natural
isomorphism $A_f : A_g \Longrightarrow A_g'$ for every generating morphism $f: g\longrightarrow g'$, 
\end{itemize} 
such that all relations described by $R$ are satisfied.
\end{proposition}

\subsection{$E_2^G$ in terms of generators and relations\label{secgenrel}}
As a preparation for the description of the little bundle operad in terms of generators and relations we briefly recall the corresponding known description for 
the little disk operad \cite[Chapter 5 and 6]{FresseI}:
One introduces the groupoid-valued operad $\PBr$ of \emph{parenthesized braids} with the  generators 

\begin{flalign}
\begin{tikzpicture}[cir/.style={circle,draw=black,inner sep=0pt,minimum size=2mm},
        poin/.style={circle, inner sep=0pt,minimum size=0mm}]
\node[poin] (0) [] at(-1,0.25) {} ;
\node[poin] (T) [] at(-1,1) {} ;
\draw[fill=black] (0) circle (0.6mm);
\draw[thick] (0) -- (T);
\node[poin] (Mouta) [label=above:{}] at (1,1) {};
\node[poin] (Min1a) [label=below:{}] at (0.7,0) {};
\node[poin] (Min2a) [label=below:{}] at (1.3,0) {};
\node[poin] (Va)  [label=left:{}] at (1,0.5) {};
\draw[thick] (Min1a) -- (Va);
\draw[thick] (Min2a) -- (Va);
\draw[thick] (Va) -- (Mouta);
\node[poin] (Moutb) [label=above:{}] at (3,1) {};
\node[poin] (Min1b) [label=below:{}] at (2.7,0) {};
\node[poin] (Min2b) [label=below:{}] at (3.3,0) {};
\node[poin] (Vb)  [label=left:{}] at (3,0.5) {};
\node[poin] (Mout'b) [label=above:{}] at (5,1) {};
\node[poin] (Min1'b) [label=below:{}] at (4.7,0) {};
\node[poin] (Min2'b) [label=below:{}] at (5.3,0) {};
\node[poin] (V'b)  [label=left:{$\tau$}] at (5,0.5) {};
\node[poin] (Ab) [label=above:{\small $c$}] at (4,0.5) {};
\draw[thick] (Min1b) -- (Vb);
\draw[thick] (Min2b) -- (Vb);
\draw[thick] (Vb) -- (Moutb);
\draw[thick] (Min1'b) -- (V'b);
\draw[thick] (Min2'b) -- (V'b);
\draw[thick] (V'b) -- (Mout'b);
\draw[dotted,->](3.5,0.5) -- (Ab) ;
\draw[dotted] (Ab) -- (4.5,0.5) ;
\node[poin] (Mout) [label=above:{}] at (7.3,1.5) {};
\node[poin] (Min1) [label=below:{}] at (6.7,0) {};
\node[poin] (Min2) [label=below:{}] at (7.3,0) {};
\node[poin] (Min3) [label=below:{}] at (7.9,0) {};
\node[poin] (V)  [label=left:{}] at (7,0.5) {};
\node[poin] (V2)  [label=left:{}] at (7.3,1) {};
\node[poin] (Mout') [label=above:{}] at (9.3,1.5) {};
\node[poin] (Min1') [label=below:{}] at (8.7,0) {};
\node[poin] (Min2') [label=below:{}] at (9.3,0) {};
\node[poin] (Min3') [label=below:{}] at (9.9,0) {};
\node[poin] (V')  [label=right:{}] at (9.6,0.5) {};
\node[poin] (V2')  [label=right:{}] at (9.3,1) {};
\node[poin] (A) [label=above:{\small $\alpha$}] at (8.3,0.5) {};
\draw[thick] (Min1) -- (V);
\draw[thick] (Min2) -- (V);
\draw[thick] (Min3) -- (V2);
\draw[thick] (V2) -- (Mout);
\draw[thick] (V) -- (V2);
\draw[thick] (Min1') -- (V2');
\draw[thick] (Min2') -- (V');
\draw[thick] (Min3') -- (V');
\draw[thick] (V2') -- (Mout');
\draw[thick] (V') -- (V2');
\draw[dotted,->](7.7,0.5) -- (A) ;
\draw[dotted] (A) -- (8.8,0.5) ; 
\node[poin] (.) [label=right:{,}] at (14,0) {};
\node[poin] (M) [label=above:{}] at (11,1) {};
\node[poin] (M1a) [label=below:{}] at (10.7,0) {};
\node[poin] (M2a) [label=below:{}] at (11.3,0) {};
\node[poin] (MVa)  [label=left:{}] at (11,0.5) {};
\draw[fill=black] (M1a) circle (0.6mm);
\draw[thick] (M1a) -- (MVa);
\draw[thick] (M2a) -- (MVa);
\draw[thick] (MVa) -- (M);
\node[poin] (MA) [label=below:{$\ell$}] at (11.8,0.5) {};
\node[poin] (M') [label=above:{}] at (13.6,1) {};
\node[poin] (M1a') [label=below:{}] at (13.3,0) {};
\node[poin] (M2a') [label=below:{}] at (13.9,0) {};
\node[poin] (MVa')  [label=left:{}] at (13.6,0.5) {};
\draw[fill=black] (M2a') circle (0.6mm);
\draw[thick] (M1a') -- (MVa');
\draw[thick] (M2a') -- (MVa');
\draw[thick] (MVa') -- (M');
\node[poin] (MB) [label=below:{$r$}] at (12.8,0.5) {};
\draw[thick] (12.3,0) -- (12.3,1); 
\draw[dotted,->](11.2,0.5) -- (MA) ;
\draw[dotted] (MA) -- (12.2,0.5) ;
\draw[dotted](12.5,0.5) -- (MB) ;
\draw[dotted,,<-] (MB) -- (13.3,0.5) ; 
\end{tikzpicture} 
\end{flalign} 
where $\tau$ denotes the application of the non-trivial permutation of 
two elements. As relations, we impose the pentagon identity for $\alpha$,
the hexagon identities for $c$ and the triangle identity on $\ell$ and $r$.
Proposition \ref{Algebras with generators and relations} 
implies that algebras over $\PBr$ are by construction braided monoidal categories. 
Figure~\ref{Fig:E2} indicates the definition of a morphism  of operads from the free operad on the depicted generators to $\Pi E_2$
which by \cite[Theorem 6.2.4]{FresseI} descends to a morphism $\PBr \to \Pi E_2$.
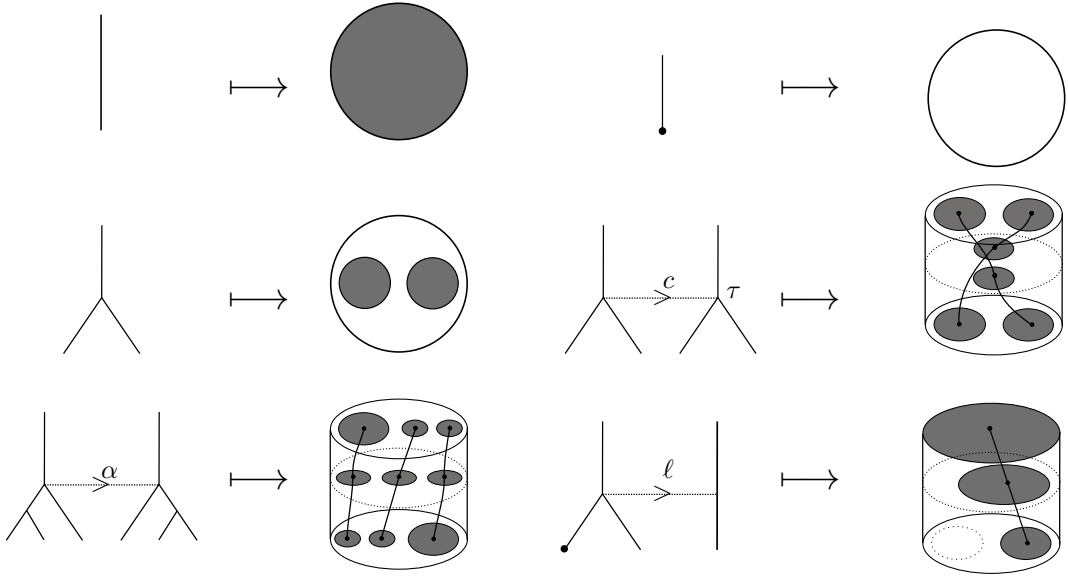
\begin{figure}[h]
\begin{center}
\begin{overpic}[scale=0.6
,tics=10]
{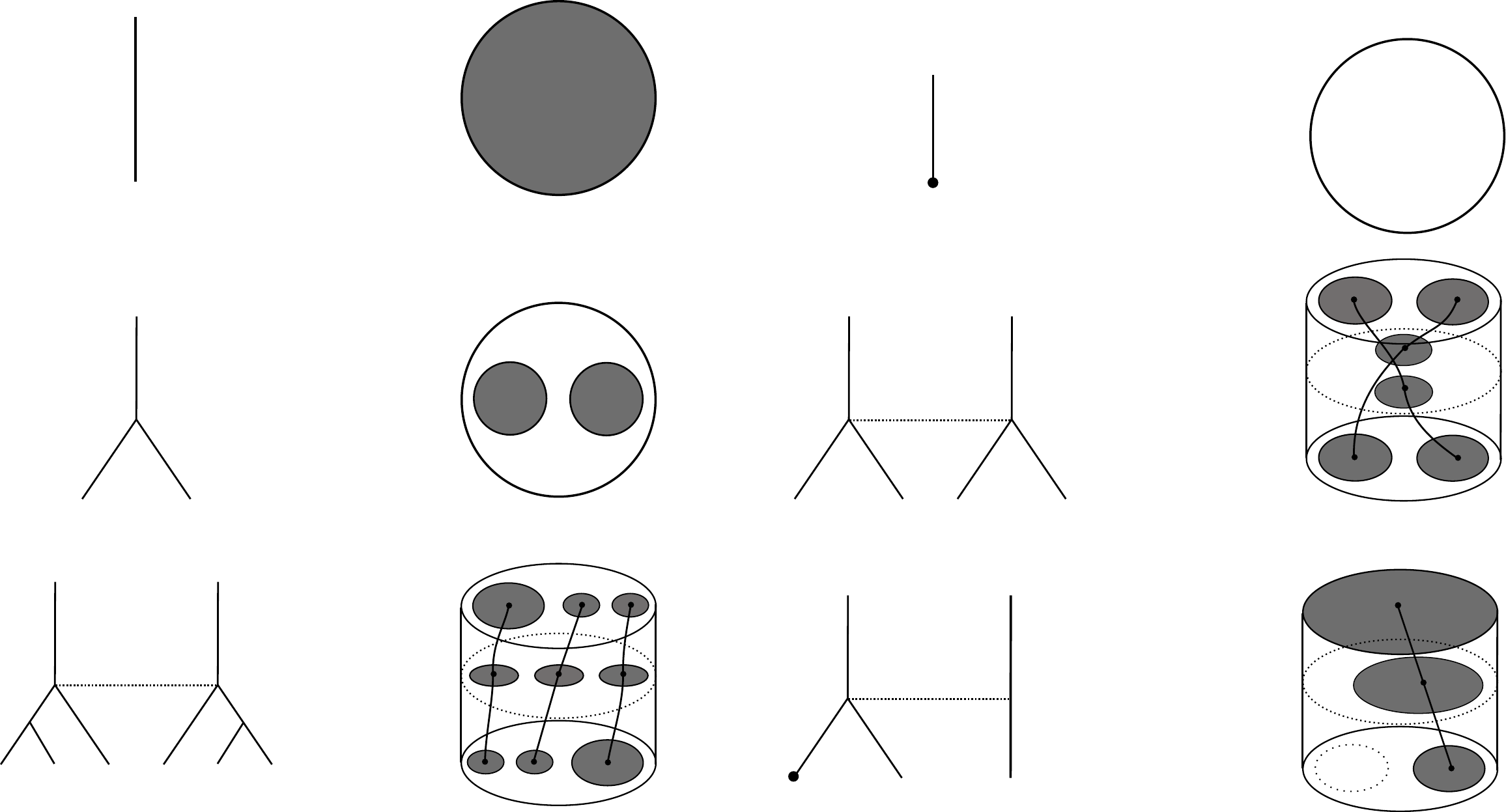}
\put(8,7.8){$>$}
\put(61,6.9){$>$}
\put(61,25.4){$>$}
\put(21,8){\Large$\longmapsto$}
\put(21,25){\Large$\longmapsto$}
\put(21,45){\Large$\longmapsto$}
\put(73,8){\Large$\longmapsto$}
\put(73,25){\Large$\longmapsto$}
\put(73,45){\Large$\longmapsto$}
\put(62,9){$\ell$}
\put(62,27){$c$}
\put(9,9){$\alpha$}
\put(68,26){$\tau$}
\end{overpic}
\vspace{0.5cm}
\caption{Definition of the morphism $\PBr\to \Pi E_2 $ on generators. We only draw the center of the circles for paths in $E_2$. The definition for $r$ is analogous to the definition for $\ell$.}\label{Fig:E2}
\end{center}
\end{figure}
By \cite[Proposition 6.2.2]{FresseI} this morphism is an equivalence, thereby giving us a presentation of $\Pi E_2$ in terms of generators and relations.

After this short review of the non-equivariant case,
we generalize this description in terms of generators and relations to the little bundle operad $E_2^G$
by introducing the $G$-colored operad $\PBr^G$ of \emph{parenthesized $G$-braids}.\label{pbrref}
Its generating operations and isomorphisms are given by

\begin{flalign}\label{Eq: Generatos EG 1}

\end{flalign}   
\begin{flalign}\label{Eq: Generatos EG 2}
%
 
\end{flalign}   
where $1$ denotes the operadic unit. 

As relations, we introduce the pentagon equation for $\alpha$, the triangle equations for $\ell$ and $r$ and, \emph{additionally}, the following relations (the labels only indicate the corresponding morphisms and $\circ$ denotes the operadic composition): 
\label{listrelationspage}
\begin{equation}
%
\end{equation} 
 
\begin{equation}\label{Eq: Gamma and Delta form an action}
%
\end{equation} 
 
\begin{equation}\label{Eq: gamma and epsilon are monoidal transformations}
%
\end{equation}

In the next step, we describe categorical algebras over $\PBr^G$. To this end, we introduce the auxiliary operad $\P^G$\label{refpgauxop} of \emph{$G$-parentheses} which differs from $\PBr^G$ by the omission of the isomorphism $c$ in \eqref{Eq: Generatos EG 1} and all the relations it is involved in. From the construction of $\P^G$ and Proposition~\ref{Algebras with generators and relations} we will be able to read off that categorical $\P^G$-algebras are a type of equivariant monoidal categories, 
called \emph{$G$-crossed monoidal categories}. This notion is based on \cite[Section~VI.1]{turaevhqft} and was further developed by other authors \cite{maiernikolausschweigerteq,Coherence}; we will use the version given in \cite[Definition 5.1]{Coherence}. In particular, our notion of a $G$-crossed monoidal category does \emph{not} include rigidity (the existence of dual objects).

Let us recall the basic definitions:
For a group $G$, consider a family $(\cat{C}_g)_{g\in G}$ of categories $\cat{C}_g$ indexed by $G$. The category $\cat{C}_g$ is often referred to as \emph{twisted sector} for $g\in G$; the sector of the neutral element $e \in G$ is also called \emph{neutral sector}. 
Suppose now that  $(\cat{C}_g)_{g\in G}$ is endowed with a homotopy coherent $G$-action shifting the sectors by conjugation, i.e.\  $g\in G$ acts as an equivalence $\cat{C}_{h} \to \cat{C}_{ghg^{-1}}$ for every $h\in G$ such that the composition of these equivalences respects the group multiplication up to coherent isomorphism.
Now a \emph{$G$-equivariant monoidal product} on $(\cat{C}_g)_{g\in G}$ consists of functors $\otimes_{g,h}:\cat{C}_g \times \cat{C}_h \to \cat{C}_{gh}$, which are associative and unital (with the unit as an object in $\cat{C}_1$) up to coherent isomorphism 
such that the $G$-action and the monoidal product intertwine up to coherent isomorphism \cite[Section 3.1]{Coherence}.
If $(\cat{C}_g)_{g\in G}$ is equipped with a homotopy coherent $G$-action shifting the sectors by conjugation and an equivariant monoidal product, we call $(\cat{C}_g)_{g\in G}$ a  \emph{$G$-crossed 
monoidal category}; it is also called \emph{$G$-equivariant monoidal category} in \cite{maiernikolausschweigerteq}. 
The definition of $\P^G$ just translates the description of $G$-crossed monoidal categories in \cite[Definition 5.1]{Coherence} into the language of operads:

\begin{lemma}
A categorical algebra over $\P^G$ is equivalent to the data of a $G$-crossed monoidal category.
\end{lemma}
\begin{proof} We use the description for algebras over operads in terms of generators
and relations given in Proposition~\ref{Algebras with generators and relations} and provide a 
concrete dictionary to the conditions in \cite{Coherence} (some of the conditions in \cite{Coherence} 
are only explicitly spelled out in their strict form, which is justified by a strictification result proven before listing the conditions; our description gives all relations in their weakest form):
The functors $\otimes_{g,h}$ correspond to the third generator in \eqref{Eq: Generatos EG 1} and the unit is given by the first generator in \eqref{Eq: Generatos EG 1}. 
The morphisms corresponding to the associativity and unitality of $\otimes_{g,h}$ correspond to the generators $\alpha$, $\ell$ and $r$. 

The homotopy coherent $G$-action is obtained from the second generator in~\eqref{Eq: Generatos EG 1} and $\gamma$ from~\eqref{Eq: Generatos EG 2}
(in \cite{Coherence} the action is denoted by $g_*$ and the natural isomorphisms $\gamma$ correspond to  $\phi^{-1}$).
The coherence conditions for the action \cite[(4) on page 123 and Equation (3.1)]{Coherence} 
correspond to the relations $(G3)$ and $(G4)$.  

The compatibility between the action and the monoidal product is encoded in the generators $\beta$ and $\varepsilon$ (the natural isomorphism corresponding to $\beta$ is called $\psi^g$ in \cite{Coherence}).
Relations $(G1)$ and $(G2)$ correspond to the condition that the action is via monoidal
functors (which is \cite[Equation (3.2) and (2) on page 123]{Coherence}). 
The further compatibility between the action and the tensor product 
is implemented via the relations $(G5)$-$(G8)$ (here $(G5)$ is \cite[(3) on page 123]{Coherence}, $(G6)$ and $(G7)$ 
are \cite[(1) on page 123]{Coherence} and
$(G8)$ is \cite[Equation (3.3)]{Coherence}). 
\end{proof}

In the sequel, we will need the following fact about $\P^G$:

	\begin{lemma}\label{lemPG}
	The operad $\P^G$ is discrete and $\pi_0 \P^G \binom{h}{ \underline{g}}$ is given by
	the set $\Sigma_r \times_h G^r$ from Proposition~\ref{propoE2Gactiongroupoid}, i.e.\ by the set of all pairs $(\sigma,\underline{b})\in \Sigma_r \times G^r$ with $r: = |\underline{g}|$
	such that
	\begin{align}
	\prod_{j=1}^r b_{\sigma(j)}       g_j b_{\sigma(j)}  = h \ .        \label{eqntheconditionongroupelements}
	\end{align}
\end{lemma}

\begin{proof}
	The $\Set$-valued operad $\pi_0 \P^G$ has the same generators as $\P^G$, but all the isomorphisms introduced in \eqref{Eq: Generatos EG 1} and \eqref{Eq: Generatos EG 2} have to be replaced by actual equalities. Recalling the generators and relations for the associative operad, we see that $\pi_0 \P^G$ is a colored version  of the associative operad where all ingoing legs can be labeled by a group element. Group elements can be pushed through the multiplication (relation $\beta$) and be composed (relation $\gamma$) according to the group law. A label by the neutral element is treated as `no label' (relation $\delta$) and a label on a leg over the unit can be deleted (relation $\varepsilon$). These relations allow to bring each of the operations in $\pi_0 \P^G \binom{h}{ \underline{g}}$ into a unique standard form where we can describe them as a pair $(\sigma,\underline{b})$ of a permutation and an $r$-tuple of group elements by arguments analogous to those in \cite[Section~4.3]{bsw}. The prescription of ingoing and outgoing colors leads to the requirement \eqref{eqntheconditionongroupelements} for $(\sigma,\underline{b})$.
	
	We still have to prove that all fundamental groups of  $\P^G \binom{h}{ \underline{g}}$ are trivial: For this we have to make sure that the given coherence diagrams for $\alpha,\beta,\gamma,\delta$ and $\varepsilon$ ensure that for each object in $\P^G \binom{h}{ \underline{g}}$ there is only one morphism starting and ending at that object. The needed arguments amount precisely to the coherence theorem for $G$-crossed monoidal categories \cite{Coherence}.
\end{proof}

A \emph{$G$-braiding} on a $G$-crossed monoidal category $(\cat{C}_g)_{g\in G}$ is a family of coherent isomorphisms
\begin{align}
X \otimes Y \cong g.Y \otimes X
\end{align}
for $X \in \cat{C}_g$ and $Y\in \cat{C}_h$. 
A $G$-crossed monoidal category equipped with a $G$-braiding is called a
\emph{braided $G$-crossed category}.
The necessary coherence conditions are given in 
\cite[Definition 5.4]{Coherence}. 

\begin{proposition}\label{propgcrossedcat}
A categorical algebra over $\PBr^G$ is equivalent to the data of a 
braided $G$-crossed category. 
\end{proposition}
\begin{proof}
We use again Proposition~\ref{Algebras with generators and relations}. The $G$-braiding is operadically captured by 
the generator $c$ in \eqref{Eq: Generatos EG 1}. The relation $(G9)$ corresponds 
to \cite[Equation (5.1)]{Coherence}, the first relation in $(G10)$ to \cite[Equation (5.3)]{Coherence} and the equation not spelled out
in detail in $(G10)$ to \cite[Equation (5.2)]{Coherence}.
\end{proof}

For $\PBr^G$ we have a description analogous to the one given for $\P^G$ in Lemma~\ref{lemPG}:
 
\begin{proposition}\label{propequivpbr}
Color-wise there is an equivalence \begin{align}  \PBr^G \binom{h}{ \underline{g}} \simeq (\Sigma_r \times_h G^r)//B_r, \quad r=|\underline{g}|\end{align} of groupoids. 
\end{proposition}

\begin{proof}
We can describe $\PBr^G$ by adding to $\P^G$ the isomorphism $c$ and the hexagon axiom that it has to satisfy. 

For each $(\sigma,\underline{b}) \in \pi_0 \P^G \binom{h}{ \underline{g}}$, see Lemma~\ref{lemPG}, and $1\le j\le r-1$ the isomorphism $c$ induces an isomorphism
\begin{align}
(\sigma,\underline{b})  \xrightarrow{\ c_{j,j+1}\ }    (\tau_{j,j+1} \sigma      ,    (b_1,\dots,b_jb_{j+1}b_j^{-1},b_j,\dots,b_r    )    )  =: c_{j,j+1}. (\sigma,\underline{b}) \ ,     \label{eqnbraction}
\end{align}
where $\tau_{j,j+1}$ is the transposition of $j$ and $j+1$.
Formally, this is achieved by choosing a standard representative for the classes, say
\begin{flalign}
\begin{tikzpicture}[scale=0.5]
		\node [style] (0) at (-15, 0) {};
		\node [style] (1) at (-16, -2) {};
		\node [style] (2) at (-14, -2) {};
		\node [style] (3) at (-14, 2) {};
		\node [style] (4) at (-12, -2) {};
		\node [style] (5) at (-13, 4) {};
		\node [style] (6) at (-12, 6) {};
		\node [style] (7) at (-11, 4) {};
		\node [style] (8) at (-10, 2) {};
		\node [style] (9) at (-9, 0) {};
		\node [style] (10) at (-8, -2) {};
		\node [style] (12) at (-11.75, 1.75) {$\dots$};
		\node [style] (13) at (-16, -3) {};
		\node [style] (14) at (-14, -3) {};
		\node [style] (15) at (-12, -3) {};
		\node [style] (16) at (-8, -3) {};
		\node [style] (17) at (-10, -2.5) {$\dots$};
		\node [style] (18) at (-15.5, -2.5) {$b_1$};
		\node [style] (19) at (-7.5, -2.5) {$b_n$};
		\node [style] (20) at (-8, -2.5) {};
		\node [style] (21) at (-13.5, -2.5) {$b_2$};
		\node [style] (22) at (-11.5, -2.5) {$b_3$};
		\node [style] (23) at (-12, -4) {};
		\node [style] (24) at (-12, -4) {$\sigma$};
		\node [style] (25) at (-6, -3.8) {,}; 
		\draw (0.center) to (1.center);
		\draw (0.center) to (2.center);
		\draw (0.center) to (3.center);
		\draw (3.center) to (4.center);
		\draw (10.center) to (6.center);
		\draw (6.center) to (3.center);
		\draw (1.center) to (13.center);
		\draw (2.center) to (14.center);
		\draw (4.center) to (15.center);
		\draw (10.center) to (16.center);
\end{tikzpicture}
\end{flalign}
applying operations in $\P^G$ to bring it into a form such that the braiding can be applied to the legs $j$ and $j+1$ and restoring the standard form by operations in $\P^G$. The $\P^G$-operations are always uniquely determined by starting point and endpoint thanks to discreteness of $\P^G$ (Lemma~\ref{lemPG}).
Now  $ \PBr^G \binom{h}{ \underline{g}}$ is equivalent to the groupoid whose objects are given by the set 
$\pi_0 \P^G \binom{h}{ \underline{g}}\cong \Sigma_r \times_h G^r$
and whose morphisms are words in the $c_{j,j+1}$ modulo the (induced) hexagon relations which -- as in the non-equivariant case -- amount precisely to the braid group relations. This proves that $ \PBr^G \binom{h}{ \underline{g}}$ is equivalent to the action groupoid of the $B_r$-action on $\Sigma_r \times_h G^r$ given by \eqref{eqnbraction}.
\end{proof}

\begin{remark}\label{remGcrossedalgebras}
	We can read off from Proposition~\ref{propequivpbr} that $\pi_0 \PBr^G$-algebras with values in vector spaces are $G$-crossed algebras as considered in \cite{kaufmannorb,turaevhqft}.
\end{remark}

Next we construct an operad morphism $\Phi : \PBr^G     \to  \Pi E_2^{G}$ 
generalizing the corresponding construction for $E_2$. 
As the underlying map of colors we use the object function $\widehat{-} : G \to \Map (\mathbb{S}^1,BG)$ of the equivalence
\begin{align}
\widehat{-} : G//G \xrightarrow{\   \simeq \ } \Pi \Map(\mathbb{S}^1,BG) \label{weakinverseeqn2}
\end{align}
from \eqref{eqnweakinverse}.

To this end, we specify the
images of the generators given in 
\eqref{Eq: Generatos EG 1} and \eqref{Eq: Generatos EG 2} (we will prove as part of Theorem~\ref{thmmain} that this assignment is compatible with the relations $(G1)$-$(G10)$):
\begin{enumerate}
	
	\item The generator
	\begin{align}\begin{tikzpicture}[cir/.style={circle,draw=black,inner sep=0pt,minimum size=2mm},
	poin/.style={circle, inner sep=0pt,minimum size=0mm}]
	\node[poin] (0) [] at(-1,1.75) {} ;
	\node[poin] (T) [label=above:{\small $e$}] at(-1,2.5) {} ;
	\draw[fill=black] (0) circle (0.6mm);
	\draw[thick] (0) -- (T);
	\end{tikzpicture}
	\end{align}
	is mapped by $\Phi$ to the embedding of an empty collection of disks (as in the non-equivariant case, see Figure~\ref{Fig:E2}) together with the constant map to $BG$ selecting the basepoint.
	
	\item 
The generator 
\begin{align}
\begin{tikzpicture}[cir/.style={circle,draw=black,inner sep=0pt,minimum size=2mm},
        poin/.style={circle, inner sep=0pt,minimum size=0mm}]
\node[poin] (0') [label=below:{\small $g$}] at(0,1.5) {} ;
\node[poin] (K) [label=right:{\small $h$}] at(0,2) {} ;
\node[poin] (T') [label=above:{\small $ hgh^{-1}$}] at(0,2.5) {} ;
\draw[thick] (0') -- (T');
\end{tikzpicture}
\end{align}
is mapped by $\Phi$ to the embedding $\disk^2 \to \disk^2 , x \mapsto x/2$ and an arbitrary choice 
\begin{align}
\disk^2\setminus \frac{\disk^2}{2} \cong \mathbb{S}^1\times\left[\frac{1}{2},1\right] \longrightarrow BG
\end{align}
of a representative in the homotopy class $\widehat{h}$ corresponding to the morphism $h : g \to hgh^{-1}$ in $G//G$.  
\label{list2}

\item 
The generator \begin{align}  \begin{tikzpicture}[cir/.style={circle,draw=black,inner sep=0pt,minimum size=2mm},
poin/.style={circle, inner sep=0pt,minimum size=0mm}]
\node[poin] (Mouta) [label=above:{\small $gh$}] at (1.5,2.5) {};
\node[poin] (Min1a) [label=below:{\small $g$}] at (1.2,1.5) {};
\node[poin] (Min2a) [label=below:{\small $h$ }] at (1.8,1.5) {};
\node[poin] (Va)  [label=left:{}] at (1.5,2) {};
\draw[thick] (Min1a) -- (Va);
\draw[thick] (Min2a) -- (Va);
\draw[thick] (Va) -- (Mouta);
\end{tikzpicture}
\end{align} 
is mapped by $\Phi$ to the embedding (see also Figure \ref{Fig:E2})
\begin{figure}
\begin{center}
\begin{overpic}[scale=0.2,
,tics=10]
{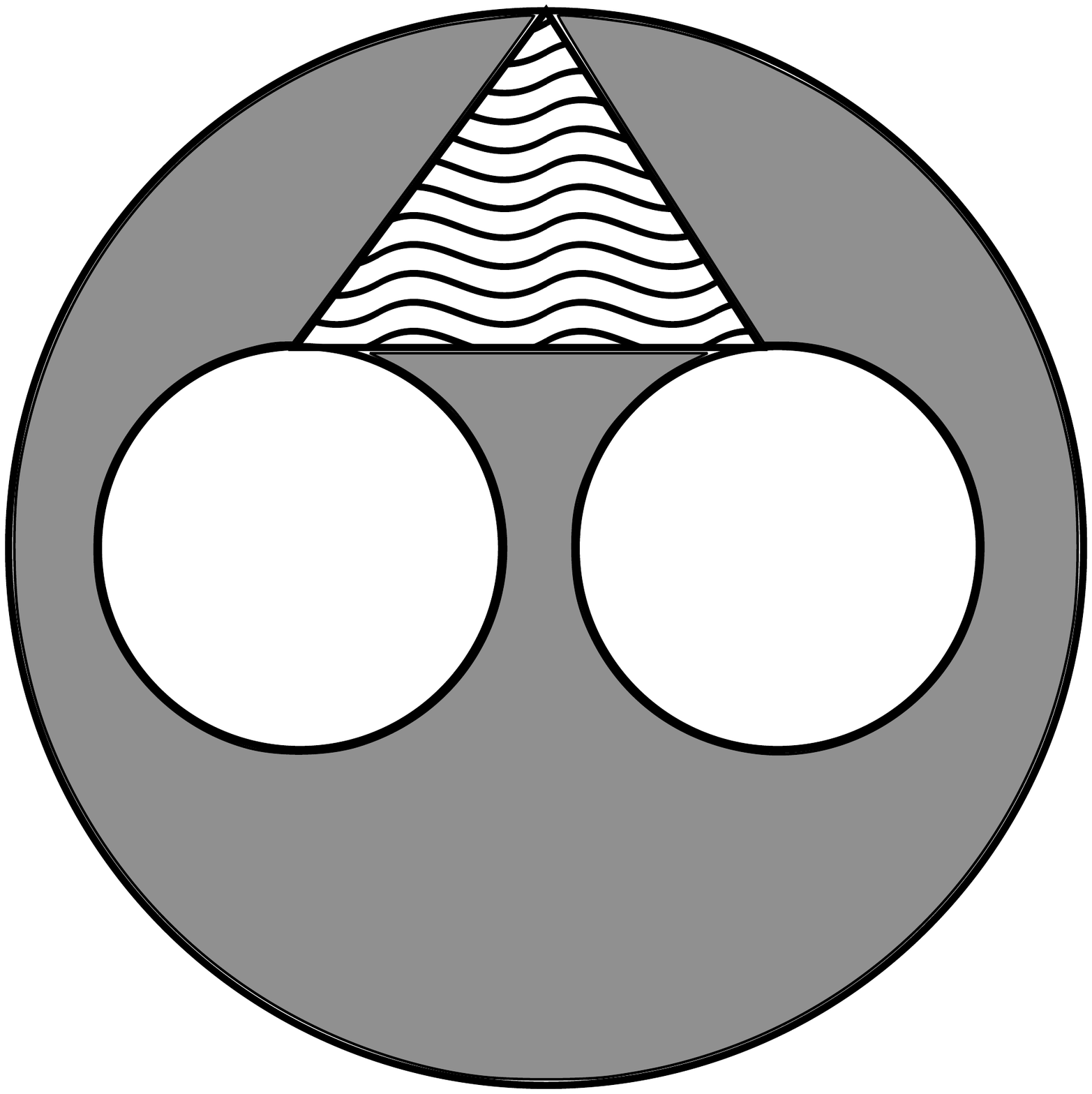}
\put(34,42){$\widehat{g}$}
\put(64,42){$\widehat{h}$}
\put(48,72){$\widehat{gh}$}
\end{overpic}
\vspace{0.5cm}
\caption{A sketch for the definition of the map $\varphi$.}\label{Fig:Product}
\end{center}
\end{figure}
\begin{align}
f:  \disk^2 \sqcup \disk^2 & \longrightarrow \disk^2 \\
x_1 & \longmapsto \frac{3}{8}\cdot x_1 - \frac{1}{2}  \left(\begin{array}{c}
0 \\ 
1
\end{array} \right) \\
x_2 & \longmapsto \frac{3}{8}\cdot x_2 + \frac{1}{2}  \left(\begin{array}{c}
0 \\ 
1
\end{array} \right)  \ \ .
\end{align} 
To equip $\comp(f)$ with a continuous map $\varphi$ to $BG$,
 we consider the decomposition of $\comp(f)$ sketched in Figure \ref{Fig:Product}. The value of $\varphi$ on the boundary is given by $\widehat{g}, \widehat{h}$ and $\widehat{gh}$. On the wavy triangle we choose $\varphi$ to be constant. Note that the gray area is homeomorphic 
to the standard  2-simplex. The 2-simplices of $BG$ are described by pairs of group elements, and we equip the gray simplex with the $BG$-valued map corresponding to $(g,h)$.   \label{list3}

\begin{figure}
\begin{center}
\begin{subfigure}[b]{0.49\textwidth}
\centering
\begin{overpic}[scale=0.6
,tics=10]
{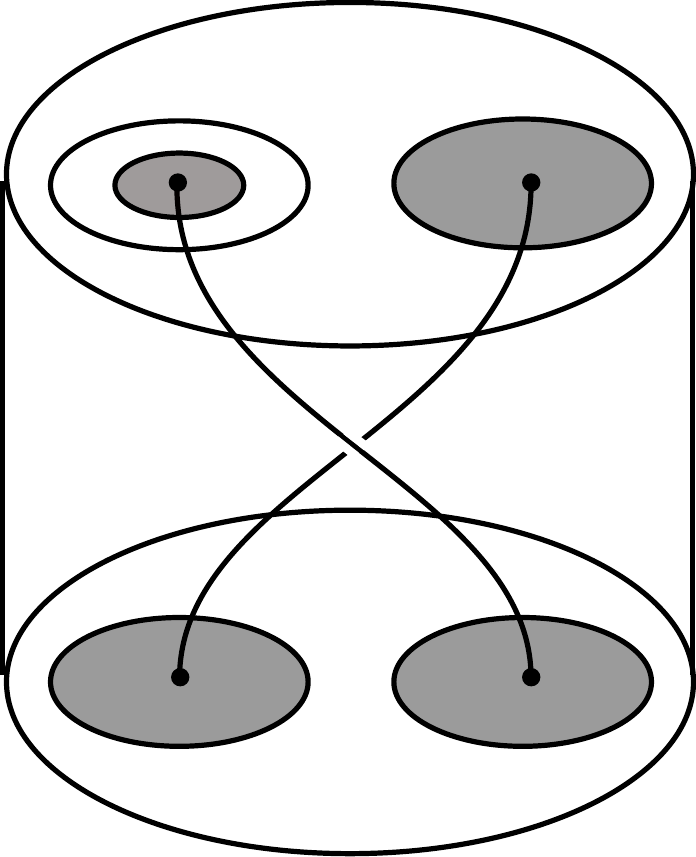}
\put(29.5,78){\small $\widehat{g}$}
\end{overpic}
\vspace{0.5cm}
\caption{$c$}\label{Fig:Braiding}
\end{subfigure}
\begin{subfigure}[b]{0.49\textwidth}
\centering
\begin{overpic}[scale=0.6
,tics=10]
{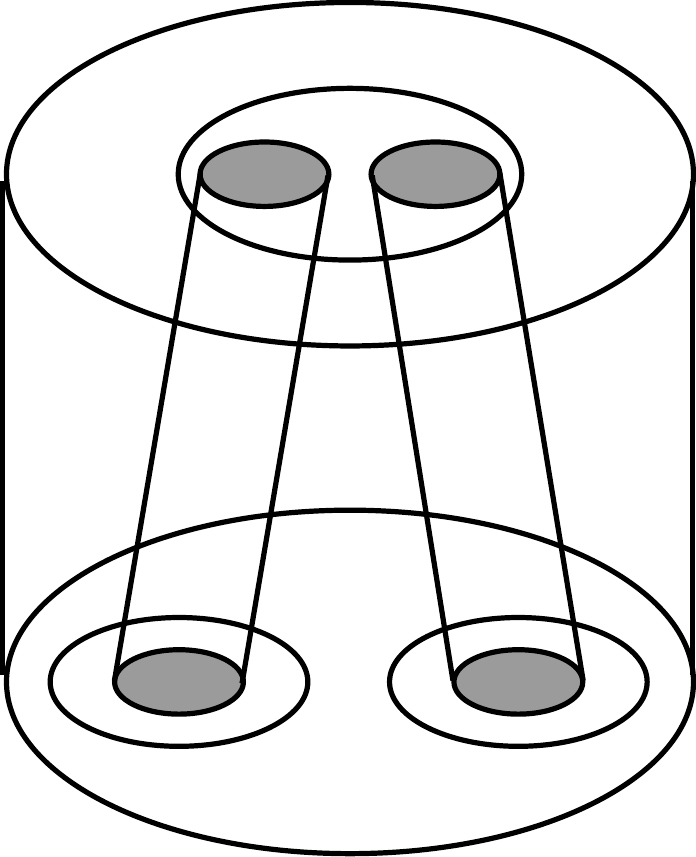}
\put(68,79){\small $\widehat{h}$}
\put(70,18){\small $\widehat{h}$}
\put(30,18){\small $\widehat{h}$}
\end{overpic}
\caption{$\beta$}\label{Fig:Beta}
\end{subfigure}
\caption{Definition of the morphism $\PBr^G\to \Pi E^G_2 $ on $c$ and $\beta$. In (a) the left circle at the bottom is labeled 
with $\widehat{g}$; the circle on the right with $\widehat{h}$. Again, we have only drawn the center of every disk when depicting paths.}
\end{center}
\end{figure}
\item The path in $E_2(2)$ underlying the $\Phi$-image of the braiding $c$ in $\Pi E^G_2\binom{\widehat{gh}}{(\widehat{g},\widehat{h})}$
is sketched in Figure~\ref{Fig:Braiding}. By Proposition
\ref{Prop: Unique lifts} there exists a lift to the fundamental groupoid of the little bundles operad which by Remark~\ref{remProp: Unique lifts}
is unique once we specify starting point and endpoint.  
The starting point is the point defined in~\ref{list3}. The endpoint is determined by the images under $\Phi$ of the generators that the target of $c$ in \eqref{Eq: Generatos EG 1} is built from and their operadic composition. 
This describes the image of $c$ in $\Pi E_2^G$.

\item 
The image of the morphisms $\ell$ and $r$ cannot be constructed using Proposition~\ref{Prop: Unique lifts} since the disks touch at the end point. However, we can use Proposition~\ref{Prop: Unique lifts} to get a path from the start point of $\ell$ and $r$ to the disk embedding $x\longmapsto x/2$ equipped with a map to $BG$ which is constant in the radial direction. Now we can rescale the disk as in Figure \ref{Fig:Delta} and leave the map constant. The composition of these two paths defines the image of $\ell$ and $r$.

\item To define the image of $\alpha$, we first define the underlying path in $E_2(3)$ to agree with the corresponding path for $\PBr$, see Figure \ref{Fig:E2}. The corresponding morphism of $\Pi E_2^G\binom{\widehat{g_1g_2g_3}}{(\widehat{g_1}, \widehat{g_2}, \widehat{g_3})}$ is again the unique lift which exists by Proposition \ref{Prop: Unique lifts} and Remark~\ref{remProp: Unique lifts}.

\begin{figure}
\begin{center}
\begin{subfigure}[b]{0.49\textwidth}
\centering
\begin{overpic}[scale=0.6
,tics=10]
{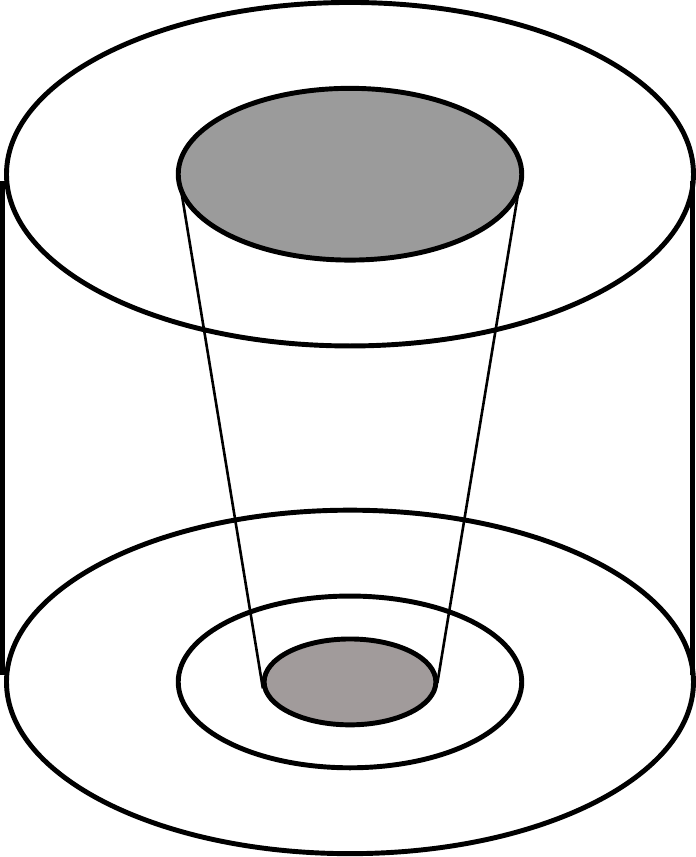}
\put(62,79){\small $\widehat{h_2h_1}$}
\put(66,20){\small $\widehat{h}_2$}
\put(53,20){\small $\widehat{h}_1$}
\end{overpic}
\vspace{0.5cm}
\caption{$\gamma$}\label{Fig:Gamma}
\end{subfigure}
\begin{subfigure}[b]{0.49\textwidth}
\centering
\begin{overpic}[scale=0.6
,tics=10]
{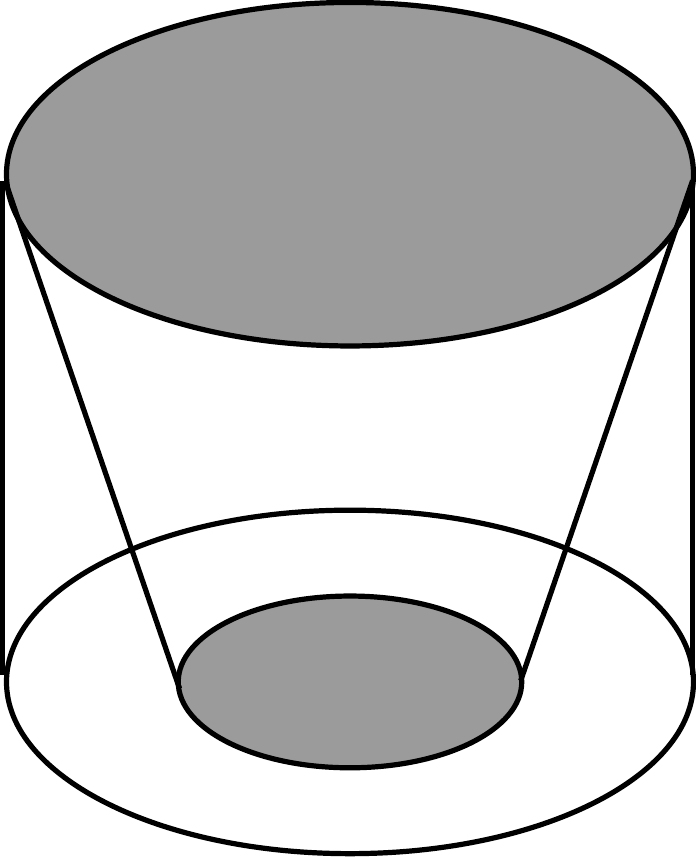}
\end{overpic}
\caption{$\delta$}\label{Fig:Delta}
\end{subfigure}
\caption{Definition of the morphism $\PBr^G\to \Pi E^G_2 $ on $\gamma$ and $\delta$.}
\end{center}
\end{figure}

\item 
The image of $\beta$ is again constructed by lifting a path in $E_2$ sketched in 
Figure~\ref{Fig:Beta}
using Proposition~\ref{Prop: Unique lifts}; for another discussion see the proof of \cite[Proposition 4.2]{extofk}.

\item The path in $E_2(1)$ underlying the image of $\gamma$ is drawn in Figure~\ref{Fig:Gamma}. We equip it with
a representative for the unique homotopy class of maps to $BG$ constructed from the unique homotopy relative boundary of maps to $BG$ between
the composition of the homotopies corresponding to $\widehat{h_1}$ and $\widehat{h_2}$ and the homotopy corresponding to 
$\widehat{h_1h_2}$.

\item The image of $\delta$ is defined by a simple rescaling sketched in Figure \ref{Fig:Delta} equipped with the constant map to $BG$.

\item To define the image of $\varepsilon$ note that if we consider a disk such that any radial path from the origin to the boundary is labeled by $g$,  the path in $BG$ corresponding to the diameter of the disk
is the composition of $\widehat{g}$ with $\widehat{g}^{-1}$ and hence homotopic to the constant map. 
We use such a homotopy to define the image of $\varepsilon$ under $\Phi$.

\end{enumerate}

\begin{theorem}\label{thmmain}
This assignment yields an equivalence 
$\Phi : \PBr^G\to \Pi E_2^G$ 
of operads in groupoids.
\end{theorem}

\begin{proof}

\begin{enumerate}
\item To show that $\Phi$ is a map of operads, we need to show that the assignments above are compatible with the relations listed on page~\pageref{listrelationspage}~ff. Verifying a relation amounts to proving that two morphisms  in components of $\Pi E_2^G$ (namely those prescribed by the above assignments) are equal. This can be achieved by observing that they have the same source and target object and that they lift the same morphism in components of $\Pi E_2$. The latter follows by construction and the fact that in the non-equivariant case
$\PBr\longrightarrow \Pi E_2$ is a map of operads. Now we invoke Remark~\ref{remProp: Unique lifts} to get the desired equality of morphisms (note that the uniqueness statement of Remark~\ref{remProp: Unique lifts} can even be used in those cases where it does not grant existence of the lifts).
Hence, we have shown that $\Phi$ descends to a morphism $\Phi:  \PBr^G \to  \Pi E_2^G $ of operads.

\item In the next step, we prove that $\Phi$ is an equivalence. First observe that $\Phi$ induces an equivalence of the categories enriched in groupoids built from $\PBr^G$ and $\Pi E_2^G$ by discarding all non-unary operations. In fact, both these categories have discrete morphism spaces and the functor induced by $\Phi$ is actually the equivalence $\widehat{-} : G//G \xrightarrow{\   \simeq \ } \Pi \Map(\mathbb{S}^1,BG)$
fixed in \eqref{eqnweakinverse}. Therefore, to conclude the proof that $\Phi$ is a equivalence, it suffices to prove that its components
\begin{align}
\Phi :      \PBr^G \binom{h}{\underline{g}}  \to E_2^G \binom{\widehat{h}}{\widehat{\underline{g}}} 
\end{align} 
are equivalences of groupoids. 
This follows from the 2-out-of-3 property because these components fit into the weakly commutative triangle
 \begin{equation}
		\begin{tikzcd}
	 & \Sigma_r \times_h G^r // B_r \ar{dl}{\simeq} \ar{dr}{\simeq}	 \\
	   \PBr^G \binom{h}{\underline{g}} \ar{rr}{\Phi}\ar{ur} & & E_2^G \binom{\widehat{h}}{\widehat{\underline{g}}} \ar{lu}
		\end{tikzcd} ,
		\end{equation}
		where $r=|\underline{g}|$. Here the equivalence $\PBr^G \binom{h}{\underline{g}}\simeq \Sigma_r \times_h G^r // B_r$ comes from Proposition~\ref{propequivpbr} and the equivalence $E_2^G \binom{\widehat{h}}{\widehat{\underline{g}}} \simeq \Sigma_r \times_h G^r // B_r$ from Proposition~\ref{propoE2Gactiongroupoid}.
\end{enumerate}
\end{proof}

An operad valued in a model category is called admissible if its category of algebras inherits a model structure in which equivalences and fibrations are created by the forgetful functor to colored objects. From \cite[Theorem~2.1]{bergermoerdijk} one can deduce that operads valued in $\Cat$
with its canonical model structure are admissible. Hence, via operadic left Kan extension $\Phi$ induces a Quillen adjunction
\begin{flalign}\label{eqnadj}
\xymatrix{
\Phi_! \,:\, \Alg (   \PBr^G   ) ~\ar@<0.5ex>[r] & \ar@<0.5ex>[l]~ \Alg (  \Pi E_2^G  ) \,:\, \Phi^*
}
\end{flalign}
between the categories of algebras over $\PBr^G$ and $\Pi E_2^G $, respectively. 
As a consequence of Theorem 4.11, we arrive at our main result:
\begin{theorem}\label{corlittlebundlesalg}
The operad map $\Phi : \PBr^G\to \Pi E_2^G$  induces a Quillen equivalence
\begin{flalign}\label{quillenequiv}
\xymatrix{
\Phi_! \,:\, \left\{ \text{braided $G$-crossed categories}  \right\} ~\ar@<0.8ex>[r]_-{\sim} & \ar@<0.8ex>[l]~ \left\{ \text{categorical little $G$-bundles algebras}  \right\} \,:\, \Phi^* \ . 
}
\end{flalign}
\end{theorem}

\begin{proof} 
Taking into account Proposition~\ref{propgcrossedcat}, we need to show that \eqref{eqnadj} is a Quillen equivalence.
By \cite[Theorem~4.1]{bergermoerdijk} this follows from $\Phi$ being an equivalence (Theorem~\ref{thmmain}) if $\PBr^G$ and $\Pi E_2^G$ are $\Sigma$-cofibrant.

To see the latter, observe $\Pi$ sends $\Sigma$-cofibrant topological operads to $\Sigma$-cofibrant categorical operads. Now $\Pi E_2^G$ is $\Sigma$-cofibrant thanks to Remark~\ref{remEnTsigmacof}. It remains to prove that $\PBr^G$ 
is $\Sigma$-cofibrant, which easily follows from the fact that the permutation action is free. 
\end{proof}

\section{Application to topological field theories}
The little bundles operad describes the genus zero part of surfaces decorated with $G$-bundles,
 hence it is intimately related to equivariant topological field theories
\cite{turaevhqft,htv,hrt} and or rather their homotopical analogues introduced in \cite{MuellerWoikeHH} using a Segal space model for the $(\infty,1)$-category $G\text{-}\Cob(n)$ 
of $G$-cobordisms based on \cite{gtmw,CalaqueScheimbauer}.  

\begin{definition}\label{defZGCob}
For any group $G$ an \emph{$n$-dimensional homotopical equivariant topological field theory with values in a symmetric monoidal $(\infty,1)$-category $\cat{S}$} is a symmetric monoidal $\infty$-functor 
\begin{align}Z: G\text{-}\Cob(n) \to \cat{S}\end{align}
\end{definition}

Informally, the objects  in $G\text{-}\Cob(2)$ are finite disjoint union of circles equipped with maps to $BG$. For two such collections of decorated circles, the morphism space between them is the space of two-dimensional compact oriented bordisms between these two collections of circles equipped with a map to $BG$ extending the ones prescribed on the boundary components.
We refer to \cite{MuellerWoikeHH} for the technical details and an example of such a field theory constructed using an equivariant version of higher derived Hochschild chains.

 We will now make the relation between the little bundles operad and topological field theory precise by proving that the value of a two-dimensional homotopical equivariant topological field theory on the circle is a homotopy little bundles algebra, thereby generalizing \cite[Proposition~6.3]{bzfn} or the earlier version of this result phrased in terms of Gerstenhaber algebras \cite{getzler}.

\begin{theorem}\label{thmequivTFTcircle}
For any homotopical two-dimensional $G$-equivariant topological field theory 
$Z$,
the values of $Z$ on the circle with varying $G$-bundle decoration combine into a homotopy algebra over the little bundles operad $E_2^G$.
\end{theorem}

A homotopy $E_2^G$-algebra is here to be understood as an algebra over the Boardman-Vogt resolution of $E_2^G$ \cite{bergermoerdijk}. 

\begin{proof}
	
	Any operation in $E_2^G     \binom{\psi}{ \underline{\varphi}}  $ can be seen as bordism $\left( \mathbb{S}^1   \right)^{\sqcup r} \to \mathbb{S}^1$, where $r:= |\underline{\varphi}|$, decorated with a map to $BG$ whose restriction to the ingoing and outgoing boundary is $\underline{\varphi}$ and $\psi$, respectively. Strictly speaking, we can see only those operations as bordisms whose $E_2$-part consists of little disks with non-intersecting boundary. This leads us to considering an equivalent suboperad of $E_2^G$ (which is not strictly unital any more, but only up to homotopy), but we will suppress this in the notation. In summary, we find maps
	\begin{align}
	E_2^G      \binom{\psi}{ \underline{\varphi}}  \to G\text{-}\Cob(2)(   (  (\mathbb{S}^1   )^{\sqcup r} ,  \underline{\varphi}   )  , (\mathbb{S}^1 , \psi)     ) \ ,     \label{homotopye2algebra0}
	\end{align}
	where $G\text{-}\Cob(n)(-,-)$ denotes the morphism spaces of $G\text{-}\Cob(2)$. The operadic composition on the left hand side is mapped to the composition of $G$-bordisms on the right hand side. 
	
	Now by definition a homotopical two-dimensional $G$-equivariant topological field theory gives us a map\begin{align}
	G\text{-}\Cob(2)(   (  (\mathbb{S}^1   )^{\sqcup r} ,  \underline{\varphi}   )  , (\mathbb{S}^1 , \psi)     ) \to [ Z(\mathbb{S}^1,\varphi_1) \otimes \dots \otimes  Z(\mathbb{S}^1,\varphi_r), Z(\mathbb{S}^1,\psi) ] \ ,        \label{homotopye2algebra}
	\end{align} 
	where $[-,- ]$ denotes the mapping space of $\cat{S}$. 
	This map, {by definition}, respects the composition up to coherent homotopy. Concatenating \eqref{homotopye2algebra0} and \eqref{homotopye2algebra} we obtain maps
	\begin{align}
	E_2^G      \binom{\psi}{ \underline{\varphi}} \to [ Z(\mathbb{S}^1,\varphi_1) \otimes \dots \otimes  Z(\mathbb{S}^1,\varphi_r), Z(\mathbb{S}^1,\psi) ] \ . 
	\end{align}
	This would endow the $\Map(\mathbb{S}^1,BG)$-colored object
	$\left(   Z(\mathbb{S}^1,\varphi)        \right)          _{\varphi \in \Map (\mathbb{S}^1,BG)}$ with an $E_2^G$-algebra structure if composition were respected strictly. 
	Instead, it is only respected up to coherent homotopy -- with the coherence data coming from $Z$.
But such a structure of an algebra over an operad respecting the operadic structure only up to coherent homotopy is precisely an (ordinary) algebra over the Boardman-Vogt resolution of that operad \cite{bergermoerdijk}.
This is also made precise in great detail in \cite[Chapter~7]{yauaqft}. 
	\end{proof}

We can use the constructions in this paper also to obtain results about ordinary (non-homotopical) 3-2-1-dimensional topological field theories with non-aspherical target space \cite{turaevhqft,extofk}. Statements about the non-aspherical case are scarce in the literature, and the following Proposition is supposed to indicate that we can make at least a statement about the value of such theories on the circle. As a target category, we choose the symmetric monoidal bicategory $\TwoVect$ of complex 2-vector spaces, see \cite{KV} for a definition.

\begin{proposition}\label{propononasphericaltarget}
	Let $T$ be a space such that $\pi_k(T)=0$ for $k\ge 3$ and $Z : T\text{-}\Cob(3,2,1) \to \TwoVect$ 
	a 3-2-1-dimensional topological field theory with target $T$ valued in the symmetric monoidal bicategory of complex 2-vector spaces.
	Then the operad $E_2^T$ takes values in 2-groupoids and the values of $Z$ on the circle combine into a homotopy $E_2^T$-algebra in 2-vector spaces. 
	\end{proposition}

\begin{proof}
	Similar arguments as those for Proposition~\ref{propolittlebundlesaspherical} show that $E_2^T$ takes values in 2-groupoids.
	Now we restrict $Z$ to a two-dimensional non-extended $(\infty,1)$-topological field theory with target $T$ and proceed as in the proof of Theorem~\ref{thmequivTFTcircle}.
\end{proof}

We will, however, not spell out the data of a (homotopy) $E_2^T$-algebra; a presentation of $E_2^T$ in terms of generators and relations is beyond the scope of this article. A first approach to $E_2^T$-algebras might be through the examples that we can produce from a cohomology class in $H^3(T;\U(1))$ using Proposition~\ref{propononasphericaltarget} 
and \cite[Theorem~3.19]{MuellerWoike1}. 

\begin{appendix}
	
	\section{Appendix}

\subsection{Properties of the auxiliary spaces $W_n^T(r)$\label{appwspaces}}
In order to investigate the spaces $W_n^T(r)$, we will need  the following construction for a pair $\langle f,\xi\rangle  \in W_n^T(r)$, i.e.\ for $f \in E_n(r)$ and a map $\xi : \comp(f) \to T$: First note that $\comp(f)$ arises from $\disk^n$ by cutting out $r$ open disks specified by their radii and centers. We now reduce each of these radii by half. Additionally, we double the radius of the outer disk. The resulting manifold with boundary is a `fattening' of $\comp(f)$ and denoted by $\widehat{\comp}(f)$,\label{fatcomp} see Figure \ref{fig: Fattening}. One can use the value of $\xi$ on the boundary of $\comp(f)$ to extend it to a map $\widehat{\xi} : \widehat{\comp}(f) \to T$.  
\begin{figure}[h]
\begin{center}
\includegraphics[scale=0.8]{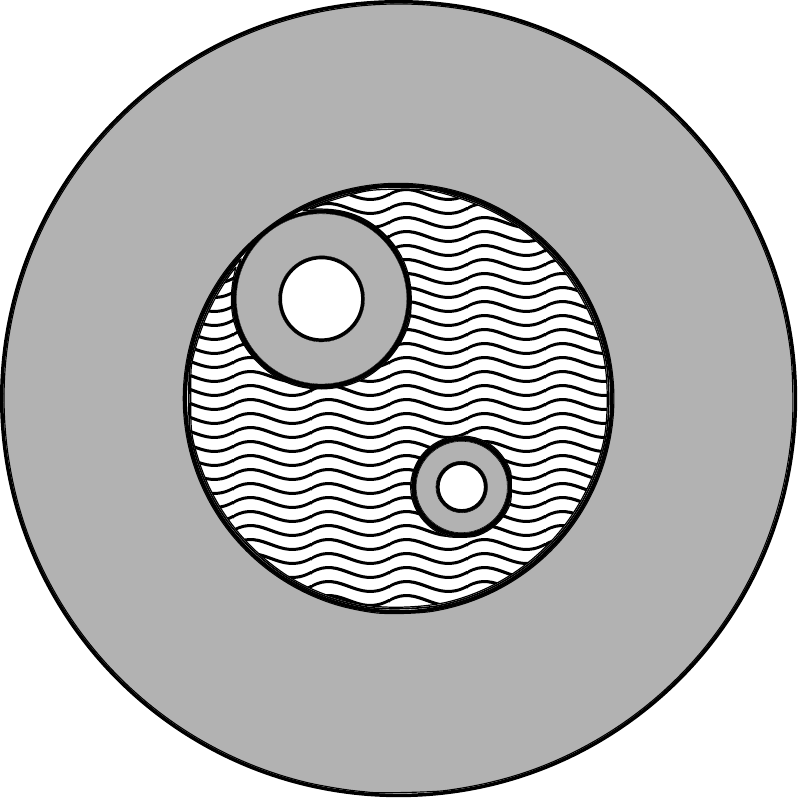}
\caption{The fattening $\widehat{\comp}(f)$ for an element $f$ in $E_2(2)$. The complement $\comp(f)$ of $f$ is the area filled with wavy lines. The space  $\widehat{\comp}(f)$ is the union of the wavy area and the gray areas.}\label{fig: Fattening}
\end{center}
\end{figure}
This extension will be referred to as \emph{radial extension}.\label{refradialext}

We define the subspace 
\begin{align}
\widehat{W}_n(r) := \{ (f,x) \in E_n(r) \times \disk^n_2 \, | \, x \in \widehat{\comp}(f)  \} \subset E_n(r) \times \disk_2^n \ ,
\end{align}
where we denote by $\disk_2^n$ the $n$-dimensional disk of radius $2$.\label{fatW}

\begin{lemma}[Continuity of radial extension]\label{Lemma: Extension}
For a topological space $Y$ let $g : Y \to W^T_n(r)$ be a map of sets satisfying condition \ref{finaltop1} and \ref{finaltop2} on page~\pageref{finaltop1}. 
Then the radial extension 
\begin{align}\label{eqnauxmap0}
\widehat{W}_n(r)\times_{E_n(r)} Y \to T , \quad 
(f,x,y)  \mapsto \widehat{g(y)}(x)
\end{align} 
is continuous.
\end{lemma}
\begin{proof}
Let $U\subset T $ be 
open and let $(f_0,x_0,y_0)$ be in its preimage under \eqref{eqnauxmap0}. We need to exhibit a neighborhood of  $(f_0,x_0,y_0)$ in $\widehat{W}_n(r)\times_{E_n(r)} Y$ whose image under \eqref{eqnauxmap0} is contained in $U$.

If $x_0 \in \overset{\circ}{\comp(f_0)}$, then $(f_0,x_0,y_0)$ is contained in the subspace ${W}_n(r)\times_{E_n(r)} Y \subset \widehat{W}_n(r)\times_{E_n(r)} Y$. By assumption, the map $W_n(r)\times_{E_n(r)}Y  \to  T $ is continuous, hence we find open neighborhoods $V$ of $f_0$ in $E_n(r)$, $V''$ of $y_0$ in $Y$ and, additionally, for some $\varepsilon>0$ an open ball \label{openballeps}$\ball_\varepsilon(x_0)$ of radius $\varepsilon>0$ around $x_0$ such that the image of the open neighborhood
$
(V' \times \ball_\varepsilon(x_0) \times V'') \cap W_n(r)\times_{E_n(r)}Y
$ of $(f_0,x_0,y_0)$ under $W_n(r)\times_{E_n(r)}Y  \to  T $ is contained in $U$. But since
\begin{align}
(V' \times \ball_\varepsilon(x_0) \times V'') \cap W_n(r)\times_{E_n(r)}Y = (V' \times \ball_\varepsilon(x_0) \times V'') \cap \widehat{W}_n(r)\times_{E_n(r)}Y
\end{align} the set $(V' \times \ball_\varepsilon(x_0) \times V'') \cap W_n(r)\times_{E_n(r)}Y$ is also an open neighborhood of $(f_0,x_0,y_0)$ in $\widehat{W}_n(r)\times_{E_n(r)}Y$ being mapped to $U$ under \eqref{eqnauxmap0}.

If $x_0\in \partial \comp(f)$, then, as in the case $x_0 \in \overset{\circ}{\comp(f_0)}$, we find suitable neighborhoods $V', V''$ and  $\ball_\varepsilon(x_0)$ such that
$(V' \times \ball_\varepsilon(x_0) \times V'') \cap W_n(r)\times_{E_n(r)}Y$ is mapped to $U$ (again by continuity of $W_n(r)\times_{E_n(r)}Y  \to  T $). By construction of the radial extension, $(V' \times \ball_\varepsilon(x_0) \times V'') \cap \widehat{W}_n(r)\times_{E_n(r)}Y$ is also mapped to $U$, which gives us the desired neighborhood is this case.

Now assume that $x_0$ is not in $\comp(f_0)$. Then the corresponding point 
on the boundary of $\comp(f)$ obtained by following a straight line in 
radial direction is also in the preimage of $U$, again by construction of the radial extension. For this point, there exists the desired open neighborhood 
as argued above. We can translate this neighborhood to $x_0$,
 rescale it to get the desired neighborhood for $x_0$
and proceed as above.   
\end{proof}

\begin{lemma}\label{lemma1}
The evaluation map $\ev : W_n(r)\times_{E_n(r)} W_n^T(r) \to T$ is continuous. 
\end{lemma}
\begin{proof}
Let $U\subset T$ be open and $(f_0,x_0 ,\xi_0)\in \ev^{-1}(U)$. We need to show that there exist an open
neighborhood of $(f_0,x_0 ,\xi_0)$  
in $W_n(r)\times_{E_n(r)} W_n^T(r)$
such that its image 
under the evaluation
is contained in
$U$. 
	
In a first step, consider the radial extension $\widehat{\xi}_0 : \widehat{\comp}(f_0) \to T$
of $\xi_0$ to the fattening of $\comp(f)$ as discussed above on page~\pageref{refradialext}.
By continuity of $\widehat{\xi}_0$ there is an $\varepsilon > 0$ such that 
 \begin{enumerate}[label=(\alph*)]
	\item 
$\ball_\varepsilon(x_0) \subset  \widehat{\comp}(f_0)  $, 

\item $\overline{\ball_\varepsilon(x_0)} \cap \partial \widehat{\comp}(f_0) = \emptyset$

\item and $\widehat{\xi}_0\left(\overline{\ball_\varepsilon(x_0)}\right)\subset U$. 
	
	\end{enumerate}
Now define the subset $U_W \subset W_n^T(r)$ of those $(f,\xi)\in W_n^T(r)$
	for which
\begin{enumerate}[label=(\alph*$'$)]
	\item $\ball_\varepsilon(x_0)\subset \widehat{\comp}(f)$, \label{condUWa}
\item $\overline{\ball_\varepsilon(x_0)} \cap \partial \widehat{\comp}(f) = \emptyset$ \label{condUWb}
\item and
$\widehat{\xi}(\overline{\ball_\varepsilon(x_0)})\subset U$. \label{condUWc}
	\end{enumerate}
Next recall that $W_n(r)\times_{E_n(r)} W_n^T(r)$ is a subspace of $  E_n(r) \times \disk_2^n \times W_n^T(r)$, where $\disk_2^n$ is the closed $n$-disk of radius 2.
The intersection of $  E_n(r) \times \ball_\varepsilon(x_0) \times U_W \subset E_n(r) \times \disk_2^n \times W_n^T(r)$ 
with $W_n(r)\times_{E_n(r)} W_n^T(r)$ contains $(f_0,x_0,\xi_0)$ and is mapped to $U$ under the evaluation. 
Hence, it remains to show that $ E_n(r) \times \ball_\varepsilon(x_0) \times U_W$ is open in $W_n(r)\times \disk_2^n\times W_n^T(r)$.
For this, it suffices to prove that $U_W$ is open in $W_n^T(r)$.

By definition of the final topology,
 a subset $V\subset W_n^T(r)$ 	is open if and only if for all maps 
	$g : Y \to W_n^T(r)$ of sets from a topological space $Y$ satisfying the conditions \ref{finaltop1} and \ref{finaltop2} 
	on page~\pageref{finaltop1} 
	the 
	preimage $g^{-1}(V)$ of $V$ is open.  	
	For the proof that $U_W$ meets this requirement,
	let $g : Y \to W_n^T(r)$ be a map 
	satisfying conditions \ref{finaltop1} and \ref{finaltop2}.
First we remark that it follows from conditions~\ref{condUWa} and \ref{condUWb} above that the image $p(U_W)$ of $U_W$ under the projection $p: W_n^T(r) \to E_n(r)$ is open in $E_n(r)$.
Hence, by \ref{finaltop1} 
the set $Y' :=  (p\circ g)^{-1}(p(U_W))\subset Y$ is also open. This implies that a subset 
of $Y'$ is open if and only if it is open seen as a subset of $Y$.
By \ref{finaltop2} the map $W_n(r)\times_{E_n(r)}Y \to T$ is 
continuous.
For this we conclude that its restriction $W_n(r)\times_{E_n(r)}Y' \to T$ is 
continuous  as well. 
It naturally gives rise to a map  
\begin{align}
\widehat{W}_n(r) \times_{E_n(r)} Y' &\to T \label{eqnauxmap} , \quad
(f,x,y) \mapsto \widehat{g(y)}(x)
\end{align} 
which is continuous by Lemma \ref{Lemma: Extension}.

There is a natural embedding 
\begin{align}
\overline{\ball_\varepsilon(x_0)} \times Y' \to \widehat W_n(r) \times_{E_n(r)}Y' , \quad (x,y)  \mapsto  (p\circ g (y), x,y)
\end{align}  
which is continuous by the universal property of the subspace and 
product topology. 
This implies that the composition
\begin{align}
 \overline{\ball_\varepsilon(x_0 )} \times Y' &\to T \\
(x,y) &\longmapsto \widehat{g(y)}(x)
\end{align}
is continuous
and hence, by adjunction, gives rise to a continuous map 
\begin{align} 
Y' \to \Map ( \overline{\ball_\varepsilon(x_0)}, T) \ . \label{prooflemma1eqn}
\end{align} 
	By definition the subspace $M(\overline{\ball_\varepsilon(x_0)},U)\subset \Map ( \overline{\ball_\varepsilon(x_0)}, T)$
	of all maps $\overline{\ball_\varepsilon(x_0)} \to T$ sending $\overline{\ball_\varepsilon(x_0)}$ to $U$  
	is open in the compact-open 
	topology. 
	This implies that the preimage of 
	$M(\overline{\ball_\varepsilon(x_0)},U)$ under \eqref{prooflemma1eqn} 
	is open. 
But this preimage is just $g^{-1}(U_W) $ showing that $g^{-1}(U_W)$ is open and finishing the proof.    
\end{proof}

	\begin{lemma}\label{lemmafibercompactopen}
		For $f\in E_n(r)$ denote by $p^{-1}(f)$ the fiber of $p:W_n^T(r) \to E_n(r)$ over $f$ endowed with the subspace topology induced from $W_n^T(r)$.
		Then the evaluation $\comp(f) \times p^{-1}(f) \to T$ is continuous and the topology on $p^{-1}(f)$ agrees with the compact-open topology, i.e.\ 
				$p^{-1}(f) = \Map(  \comp(f), T)$. 
	\end{lemma} 
	
	\begin{proof}
		The map $\comp(f) \times p^{-1}(f) \to T$ is the restriction of the map $\ev : W_n(r)\times_{E_n(r)} W_n^T(r) \to T$ from Lemma~\ref{lemma1} to the fiber of $W_n(r)\times_{E_n(r)} W_n^T(r) \to E_n(r)$ over $f$ and hence continuous.
		
		Moreover, $p^{-1}(f) = \Map(  \comp(f), T)$ as sets, so it remains to show that the identity map is continuous in both directions.
		
		The identity is continuous as a map $\Map(  \comp(f), T) \to p^{-1}(f)$: For this it suffices to show that the composition $\Map(  \comp(f), T) \to W_n^T(r)$ 
		with the inclusion $p^{-1}(f) \to W_n^T(r)$ is continuous. The composition of $\Map(  \comp(f), T) \to W_n^T(r) \to E_n(r)$ factors through $\{f\} \to E_n(r)$ and is therefore continuous, so condition~\ref{finaltop1} is fulfilled. For condition~\ref{finaltop2} to be fulfilled, we need  the evaluation map
		$\comp(f) \times \Map(  \comp(f), T)  \to  T$ to be continuous. But this is the case because $\comp(f)$ is locally compact.
		
		The identity is continuous as a map $p^{-1}(f) \to \Map(  \comp(f), T)$: By adjunction (and since $\comp(f)$ is locally compact), continuity of $p^{-1}(f) \to \Map(  \comp(f), T)$ is equivalent to continuity of $\comp(f) \times p^{-1}(f) \to T$, which we have already established.
	\end{proof}

\begin{proposition}\label{appLemma: Serre fibration}
The map $p : W_n^T(r) \to E_n(r) $ is a Serre fibration. 
\end{proposition}
\begin{proof}
Let us fix $f_0 \in E_n(r)$ and an $\varepsilon>0$ smaller than the radii of all 
disks in the image of $f_0$. Let $X:=  \bigcup_{i=1}^r  \ball_\varepsilon(c_i) $ be the 
union of $\varepsilon$-balls around the centers $c_1,\dots,c_r$ of $f_0$.
We define an open neighborhood $U_\varepsilon$ of $f_0$ in $E_n(r)$ consisting of those $f \in E_n(r)$  satisfying the following requirements (illustrated in Figure~\ref{fig: different center}): 
\begin{itemize}
\item $X \cap \comp(f)=\emptyset$. 

\item The center of any disk belonging to $f$ is contained in $X$ (by the first requirement each center of $f$ is contained in $\ball_\varepsilon(c_i)$ for a unique $i$).
\end{itemize}

\begin{figure}[h]
	\begin{center}
		\includegraphics[scale=2]{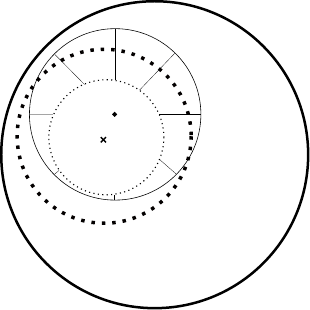}
	\end{center}
	\caption{For the proof of Proposition~\ref{appLemma: Serre fibration}. The large dashed circle corresponds to $f_0$, the solid circle to $f$ and the small dashed circle to a circle of radius $\varepsilon$ around the center of $f_0$ represented by a cross. The solid lines indicate the direction of the radial extension.}\label{fig: different center}
\end{figure}
We set $\comp(\varepsilon,f_0):=\disk^n\setminus X$.
Since being a Serre fibration is a local property \cite[Theorem~VII.6.11]{bredon}, it suffices to prove that for $m\ge 0$ the lifting problem
  \begin{equation}
\begin{tikzcd}
\disk^m \times 0  \ar{rr}{K} \ar[swap]{dd}{} & & p^{-1} (U_\varepsilon) \ar{dd}{p} \\
  	& & \\
 \disk^m \times I  \ar{rr}{L} \ar[dashrightarrow]{rruu}{\widetilde L} & & U_\varepsilon
\end{tikzcd} 
\end{equation}
can be solved.  
For every $x\in \disk^m$ 
we write $K(x) = (K'(x),K''(x))$, where $K'(x) \in E_n(r)$ and $K''(x) : \comp(K'(x)) \to T$.
We can continuously extend $K''(x)$  radially in the direction of the centers of $K'(x)$ to a map $\widetilde K''(x) :  \comp(\varepsilon,f_0)\to T$ (constantly along the radial lines in Figure~\ref{fig: different center}).
Since $\comp(L(x,t)) \subset \comp(\varepsilon,f_0)$,
 we can set $\widetilde{L}(x,t):= (L(x,t),\widetilde K''(x)|_{\comp(L(x,t))} )$ for $x \in \disk^m$ and $t\in I$.
This is obviously a $p$-lift of $L$ as a map of sets. 
It remains to show that $\widetilde{L}$ is continuous: Indeed, condition~\ref{finaltop1} is satisfied by definition. 
For \ref{finaltop2} we investigate the map \begin{align} W_n(r) \times_{E_n(r)} (\disk^m\times I) \to T \label{eqncriticalmap2} \end{align} and realize that its domain is a subspace
\begin{align} W_n(r) \times_{E_n(r)} (\disk^m\times I) \subset  U_\varepsilon \times \comp(\varepsilon,f_0) \times \disk^m\times I \end{align} and that \eqref{eqncriticalmap2} is the restriction of
\begin{align}
  U_\varepsilon \times \comp(\varepsilon,f_0) \times \disk^m\times I \xrightarrow{\ \operatorname{pr}\ }   U_\varepsilon \times \comp(\varepsilon,f_0) \times \disk^m \xrightarrow{\ F\ } T \ ,
\end{align}
where 
\begin{align}
F:  U_\varepsilon \times \comp(\varepsilon,f_0) \times \disk^m &\to T \label{eqnFmap} \\
(f,x,x')& \mapsto (\widetilde{K}''(x'))(x)
\end{align}
Hence, it suffices to prove that $F$ is continuous.
This follows from a slight modification of the proof of Lemma~\ref{Lemma: Extension} showing that
the radial extension used here is continuous as well. 
\end{proof}

We also need to prove the following statement about the map $q$ from \eqref{therestrictionqeqn}: 

\begin{proposition}\label{applemmaqisfibration}
  	The map $q: W_n^T(r) \to \prod^{r+1} \Map (\mathbb{S}^{n-1},T)$ 
  	obtained by restriction to the boundary 
  	is a Serre fibration.
  	\end{proposition}

\begin{proof}
We need to prove that for $m\ge 0$ the lifting problem
\begin{equation}
\begin{tikzcd}
\disk^m \times 0  \ar{rr}{K} \ar[swap]{dd}{} & & W_n^T(r) \ar{dd}{q} \\
& & \\
\disk^m \times I  \ar{rr}{L} \ar[dashrightarrow]{rruu}{\widetilde L} & & \prod^{r+1} \Map (\mathbb{S}^{n-1},T)
\end{tikzcd} 
\end{equation}
can be solved. 
For this we write $K(x) = (K'(x),K''(x))$ for $x\in \disk^m$, where $K'(x) \in E_n(r)$ and $K''(x) \in \Map (\comp (K'(x)),T)$. Next note that
$L$ gives us for each $x\in \disk^m$ paths $h_1^x, \dots,h_{r+1}^x$ in $\Map (\mathbb{S}^{n-1},T)$, and for $1\le j\le r+1$ the path $h_j^x$ is a homotopy of maps $\mathbb{S}^{n-1}\to T$ 
 starting at the $j$-th component $q_j (K''(x))$ of the restriction of $K''(x) : \comp (K'(x)) \to T$ to the boundary of $\comp(K'(x))$. 

The desired lift $\widetilde L : \disk^m \times I \to W_n^T(r)$ can now be described as follows: For $(x,t) \in \disk^m \times I$ the $E_n(r)$-part of $\widetilde L(x,t)$ is obtained from $K'(x)$ by enhancing the radius of the outer disk by $t$ and reducing the radii of the inner disks by multiplying them by $1-t/2$. Afterwards, we rescale by the factor $1/(1+t)$ to really obtain a point in $E_n(r)$. The needed map from the complement of this point in $E_n(r)$ to $T$ is obtained by gluing together $K''(x)$ and the restriction of the homotopies $h_1^x, \dots , h_{r+1}^x$ to $[0,t]$. 
\end{proof}

The spaces $W_n^T$ allow for a gluing map that we need in order to prove that the composition of the little bundles operad is continuous:

\begin{lemma}\label{lemmagluingmaps}
Let $r$ be a positive integer. 
For $1\leq j \leq r$ the gluing map
	\begin{align}
	\widehat{\circ}_j : W_n^T(r) \times_{  \Map(\mathbb{S}^{n-1},T)  } W_n^T(r') \to W_n^T (r+r'-1), \ \left(   \langle f,\xi\rangle  ,  \langle f',\xi'  \rangle \right) \mapsto \langle  f \circ_j f'  ,  \xi \cup_{\mathbb{S}^{n-1}}^j \xi' \rangle  
	\end{align} 
	is continuous.
	Here $f \circ_j f'$ is the operadic composition in $E_n$, and $\xi \cup_{\mathbb{S}^    {n-1} }^j \xi'$ is the map obtained from gluing $\xi$ and $\xi'$ along the $j$-th sphere $\mathbb{S}^{n-1}$ in the domain of definition of $\xi$.
	\end{lemma}

\begin{proof}
	By definition of the topology of the spaces $W_n^T$,
	 continuity of the gluing map
	  amounts to proving that the composition with $W_n^T(r+r'-1)\to E_n(r+r'-1)$ is continuous (which is obvious)
	and that the evaluation
	\begin{align}
	W_n(r+r'-1) \times_{E_n(r+r'-1)} \left(  W_n^T(r) \times_{  \Map(\mathbb{S}^{n-1},T)  } W_n^T(r')    \right) \to T \label{eqnevacontproof}
	\end{align}
	is continuous. 
	
	The latter can be seen by factorizing \eqref{eqnevacontproof} into continuous maps. As a first observation, we describe the left hand side of \eqref{eqnevacontproof} as the pushout
\begin{equation}
\begin{footnotesize}
\begin{tikzcd}
 \mathbb{S}^{n-1}\times W_n^T(r) \times_{\Map(\mathbb{S}^{n-1},T)}   W_n^T(r') \ar{r} \ar{dd}  &  W_n(r') \times_{E_n(r')} W_n^T(r) \times_{\Map(\mathbb{S}^{n-1},T)}   W_n^T(r')  \ar{dd} 
 \\ 
& &  \\ 
 W_n(r) \times_{E_n(r)} W_n^T(r) \times_{\Map(\mathbb{S}^{n-1},T)}   W_n^T(r') \ar{r} 
 &  \substack{ \displaystyle    \left(  W_n(r) \times_{E_n(r)} \left(  W_n^T(r) \times_{\Map(\mathbb{S}^{n-1},T)}   W_n^T(r')  \right)  \right)   \\ \displaystyle  \cup \left(    W_n(r') \times_{E_n(r')} \left( W_n^T(r) \times_{\Map(\mathbb{S}^{n-1},T)}   W_n^T(r')  \right)  \right)}
\end{tikzcd}
\end{footnotesize}
\end{equation}
	where the map 
	\[\mathbb{S}^{n-1}\times E_n(r)\times E_n(r')\times W_n^T(r) \times_{\Map(\mathbb{S}^{n-1},T)}   W_n^T(r') \to  W_n(r) \times_{E_n(r)} W_n^T(r) \times_{\Map(\mathbb{S}^{n-1},T)}   W_n^T(r')\] first projects to $\mathbb{S}^{n-1}\times E_n(r)\times W_n^T(r) \times_{\Map(\mathbb{S}^{n-1},T)}   W_n^T(r')$ and then identifies $\mathbb{S}^{n-1}$ with the outer boundary sphere and the map
	\[\mathbb{S}^{n-1}\times E_n(r)\times E_n(r')\times W_n^T(r) \times_{\Map(\mathbb{S}^{n-1},T)}   W_n^T(r') \to  W_n(r') \times_{E_n(r')} W_n^T(r) \times_{\Map(\mathbb{S}^{n-1},T)}   W_n^T(r')\] first projects to $\mathbb{S}^{n-1}\times E_n(r')\times W_n^T(r) \times_{\Map(\mathbb{S}^{n-1},T)}   W_n^T(r')$ and identifies $\mathbb{S}^{n-1}$ with the $j$-th ingoing boundary sphere.
Indeed, there is a homeomorphism
\begin{align}
&	W_n(r+r'-1) \times_{E_n(r+r'-1)} \left(  W_n^T(r) \times_{  \Map(\mathbb{S}^{n-1},T)  } W_n^T(r')    \right) \\   \xrightarrow{\ \cong\ } &  \left(  W_n(r) \times_{E_n(r)} W_n^T(r) \times_{\Map(\mathbb{S}^{n-1},T)}   W_n^T(r')   \right)      \cup \left(    W_n(r') \times_{E_n(r')} W_n^T(r) \times_{\Map(\mathbb{S}^{n-1},T)}   W_n^T(r')  \right)
\end{align}
given by
\begin{align}
\left(    (f_1 \circ f_2 , x) , \langle f_1,\xi_1\rangle , \langle f_2,\xi_2\rangle  \right) \longmapsto 
\begin{cases}
\left(    (f_1  , x) , \langle f_1,\xi_1\rangle , \langle f_2,\xi_2\rangle  \right) \ , \ \text{ if $x\in \comp(f_1)$} \\ 
\left(    (f_2  , x) , \langle f_1,\xi_1\rangle , \langle f_2,\xi_2\rangle  \right) \ , \ \text{ if $x\in \comp(f_2)$} \\
\end{cases}
\end{align}
	Now \eqref{eqnevacontproof} is the composition of continuous maps
		\begin{equation}
	\begin{tikzcd}
	W_n(r+r'-1) \times_{E_n(r+r'-1)} \left(  W_n^T(r) \times_{  \Map(\mathbb{S}^{n-1},T)  } W_n^T(r')    \right) \ar{d}{\cong} \\  \left(  W_n(r) \times_{E_n(r)} W_n^T(r) \times_{\Map(\mathbb{S}^{n-1},T)}   W_n^T(r')   \right)      \cup \left(    W_n(r') \times_{E_n(r')} W_n^T(r) \times_{\Map(\mathbb{S}^{n-1},T)}   W_n^T(r')  \right) \ar{d}{\text{projection}}\\
	\left(  W_n(r) \times_{E_n(r)} W_n^T(r)\right)   \cup     \left(   W_n(r')    \times_{E_n(r')}    W_n^T(r')  \right) \ar{d}{\text{evaluation}}\\T \ .
	\end{tikzcd} 
	\end{equation} 
	\end{proof}

\subsection{Computation of (some) homotopy colimits}
In this Appendix we present the computation of a homotopy colimit 
needed for the description of the groupoid model of the little bundles operad.

First recall that for any diagram $X$ from a groupoid $\Omega$ to spaces, the homotopy colimit is given by the realization of the simplicial space with level $n$ given by \begin{align}
\coprod_{\vec{y} : [n]     \to \Omega}     X(y_0) \ , 
\end{align}
where the coproduct runs over all strings $\vec{y} : [n]     \to \Omega$ of length $n\ge 0$,
see e.g.\  \cite[Corollary~5.1.3]{riehl} for the definition of the face and degeneracy maps.

\begin{lemma}\label{lemhocolimexp}
	Let $\Gamma$ be a diagram from a groupoid $\Omega$ to groupoids.
	The homotopy colimit of $B\Gamma$
	is the nerve of a groupoid admitting the following description:
	\begin{itemize}
		\item The objects are given by pairs $(y_0,x_0)$, where $y_0 \in \Omega$ and $x_0 \in \Gamma_0(y_0)$. 
		
		\item For every pair $(g_0,f_0)$, where $g_0 : y_0 \to y_1$ is a morphism in $\Omega$ and $f_0 : x_0 \to x_1$ a morphism in $\Gamma(y_0)$, we get a morphism $(y_0,x_0) \to (y_1,g_0.x_1)$, where $g_0.x_1=\Gamma(g_0)(x_1)$.

		\item	The composition of morphisms is given by
		\begin{align}
		\left(  y_1 \xrightarrow{\  g_1    \ } y_2 ,g_0. x_1 \xrightarrow{\  f_1    \ } x_2           \right) \circ \left(  y_0 \xrightarrow{\  g_0    \ } y_1 , x_0 \xrightarrow{\  f_0    \ } x_1           \right) := 	\left(   y_0 \xrightarrow{\  g_1g_0    \ } y_1       , x_0 \xrightarrow{\ (g_0^{-1} .   f_1) f_0     \ } g_0^{-1}x_2          \right)
		\end{align} \label{lemhocolimexp1}
	\end{itemize}

\end{lemma}

\begin{proof}
	As just explained, the desired homotopy colimit is the realization of the simplicial space with 
	\begin{align}
	\coprod_{\vec{y} : [n]     \to \Omega}     B\Gamma (y_0) 
	\end{align}
	in level $n$. Since the realization can be computed as the diagonal, we find
	\begin{align}
	\left(	\underset{\Omega}{\hocolim}\, B\Gamma  \right)_n = 	\coprod_{\vec{y} : [n]     \to \Omega}     B_n\Gamma (y_0)  \ .
	\end{align}
	Carefully writing out the low degree face and degeneracy maps yields the claim.
\end{proof}

\end{appendix}

\end{document}